\documentclass[10pt]{amsart}

\usepackage{amsmath,amssymb,amsthm,amsfonts,mathrsfs}
\usepackage{amscd,stmaryrd}
\usepackage[usenames,dvipsnames]{color}
\usepackage{hyperref}
\usepackage{graphicx,subfigure}
\definecolor{nicegreen}{RGB}{0,180,0}
\hypersetup{colorlinks=true,citecolor=nicegreen,linkcolor=Red,urlcolor=blue,pdfstartview=FitH}

\usepackage[divide={2.45cm,*,2.45cm}]{geometry}
\usepackage{enumerate}
\usepackage{color}
\usepackage[all]{xy}
\usepackage{bm}
\usepackage{tikz-cd}

\allowdisplaybreaks

\newtheorem{thm}{Theorem}[section]
\newtheorem{cor}[thm]{Corollary}
\newtheorem{lemma}[thm]{Lemma}
\newtheorem{propn}[thm]{Proposition}
\newtheorem*{nothm}{Theorem}

\newtheorem*{nopropn}{Proposition}

\theoremstyle{remark}

\theoremstyle{definition}
\newtheorem{remark}[thm]{Remark}
\newtheorem{ex}[thm]{Example}

\newcommand{\qp}{\mathbb{Q}_p}

\newcommand{\fpb}{\overline{\mathbb{F}}_p}

\newcommand{\T}{\textnormal{T}}
\newcommand{\aff}{\textnormal{aff}}
\newcommand{\mm}{\mathfrak{m}}
\newcommand{\nn}{\mathfrak{n}}
\newcommand{\vv}{\mathfrak{v}}
\newcommand*{\longhookrightarrow}{\ensuremath{\lhook\joinrel\relbar\joinrel\rightarrow}}
\newcommand*{\longtwoheadrightarrow}{\ensuremath{\relbar\joinrel\twoheadrightarrow}}
\newcommand{\sub}[2]{\genfrac{}{}{0pt}{}{#1}{#2}}

\newcommand{\End}{\textnormal{End}}

\newcommand{\cind}{\textnormal{c-ind}}
\newcommand{\Hom}{\textnormal{Hom}}
\newcommand{\Ext}{\textnormal{Ext}}

\newcommand{\tor}{\textnormal{tor}}
\newcommand{\free}{\textnormal{free}}
\newcommand{\pd}{\textnormal{pd}}
\newcommand{\id}{\textnormal{id}}
\newcommand{\bld}{\pmb{\mathbf{1}}}

\newcommand{\cA}{{\mathcal{A}}}
\newcommand{\cB}{{\mathcal{B}}}
\newcommand{\cC}{{\mathcal{C}}}

\newcommand{\cF}{{\mathcal{F}}}

\newcommand{\cH}{{\mathcal{H}}}

\newcommand{\cP}{{\mathcal{P}}}

\newcommand{\cR}{{\mathcal{R}}}

\newcommand{\cY}{{\mathcal{Y}}}

\newcommand{\bbI}{{\mathbb{I}}}

\newcommand{\bbL}{{\mathbb{L}}}

\newcommand{\bbQ}{{\mathbb{Q}}}
\newcommand{\bbR}{{\mathbb{R}}}

\newcommand{\bbZ}{{\mathbb{Z}}}

\newcommand{\tW}{\widetilde{W}}

\newcommand{\tOmega}{\widetilde{\Omega}}

\newcommand{\tu}{\widetilde{u}}

\newcommand{\tw}{\widetilde{w}}

\begin{document}

\title{Homological dimension of simple pro-$p$-Iwahori--Hecke modules}
\author{Karol Kozio\l}
\address{Department of Mathematics \\ University of Toronto \\ Toronto, ON M5S 2E4, Canada} 
\email{karol@math.toronto.edu}

\begin{abstract}
Let $G$ be a split connected reductive group defined over a nonarchimedean local field of residual characteristic $p$, and let $\mathcal{H}$ be the pro-$p$-Iwahori--Hecke algebra over $\overline{\mathbb{F}}_p$ associated to a fixed choice of pro-$p$-Iwahori subgroup.  We explore projective resolutions of simple right $\mathcal{H}$-modules.  In particular, subject to a mild condition on $p$, we give a classification of simple right $\mathcal{H}$-modules of finite projective dimension, and consequently show that ``most'' simple modules have infinite projective dimension.  
\end{abstract}

\maketitle

\section{Introduction}

Let $p$ be a prime number.  The mod-$p$ representations of $p$-adic reductive groups have been the subject of intense recent study, culminating in the work of Abe--Henniart--Herzig--Vign\'eras \cite{ahhv}.  The aforementioned article classifies irreducible admissible representations of a $p$-adic reductive group $G$ in terms of supersingular representations, and it is expected that this classification will be useful in extending the mod-$p$ Local Langlands Correspondence beyond the case $G = \textnormal{GL}_2(\qp)$.

In a parallel vein, there has been substantial progress in recent years in understanding the category of modules over the \emph{pro-$p$-Iwahori--Hecke algebra} $\cH = \fpb[I(1)\backslash G/I(1)]$, where $I(1)$ is a fixed pro-$p$-Iwahori subgroup of $G$.  We refer the reader to \cite{vigneras:hecke1} for a description of this algebra, and to \cite{vigneras:hecke3}, \cite{ollivier:compat}, \cite{abe:heckemods} for a classification of its simple modules.  The interest in this algebra comes from its link with the category of smooth mod-$p$ representations of $G$, and thus with the mod-$p$ Local Langlands Program.  To wit, there exists an equivalence between the category of $\cH$-modules and the category of smooth mod-$p$ $G$-representations generated by their $I(1)$-invariant vectors, when $G = \textnormal{PGL}_2(\qp)$ or $\textnormal{SL}_2(\qp)$ (\cite{ollivier:foncteur}, \cite{koziol:glnsln}).  Moreover, if one replaces $\cH$ by a certain related differential graded Hecke algebra $\cH^\bullet$, one obtains an equivalence between the (unbounded) derived category of differential graded $\cH^\bullet$-modules and the (unbounded) derived category of $G$-representations (subject to some restrictions on $I(1)$, see \cite{schneider:dga}).

Since we have a good understanding of the structure of simple $\cH$-modules, it therefore becomes imperative to better understand their homological properties.  In the case when $G$ is split, a first step towards achieving this was taken by Ollivier and Schneider in \cite{ollivierschneider}, where a functorial resolution of any $\cH$-module was constructed by making use of coefficient systems on the Bruhat--Tits building of $G$.  The authors use this resolution to show, among other things, that the algebra $\cH$ has infinite global dimension (at least generically), and possesses a simple module of infinite projective dimension.

Our goal in this note will be to give a classification of those simple $\cH$-modules of finite projective dimension.  We now give an overview of the contents herein, and of our main result.  

\vspace{\baselineskip}

We assume henceforth that the group $G$ is defined and split over a fixed nonarchimedean local field $F$ of residual characteristic $p$.  After recalling the necessary notation in Section \ref{notation}, we investigate the algebras $\cH_{\cF}$ and $\cH_{\cF}^\dagger$, which are certain ``small'' subalgebras of $\cH$ associated to a facet $\cF$ in the semisimple Bruhat--Tits building of $G$.  The Ollivier--Schneider resolution is constructed from algebras of this form, and the projective dimensions of $\cH$-modules are controlled by their restrictions to $\cH_{\cF}$ and $\cH_{\cF}^\dagger$.  Thus, it will be important for our purposes to understand how to transfer homological properties of modules from one algebra to the other.  We take this up in Section \ref{prelim}.  Once this is complete, we recall more precisely in Section \ref{resnsandsss} the resolution of $\cH$-modules constructed in \cite{ollivierschneider}, and record a useful associated spectral sequence.

With the preliminaries in place, we review in Section \ref{simplemods} Abe's classification of simple $\cH$-modules in terms of parabolic coinduction and supersingular modules of Levi components (see \cite{abe:heckemods}).  Lemma \ref{cech} shows that every simple module admits a resolution by a certain \v{C}ech complex, with each term being a direct sum of parabolically coinduced representations.  Using the results already obtained, along with an analog of the Mackey formula for $\cH$-modules (Lemma \ref{mackey}), we obtain the following result (Proposition \ref{Anprop}):

\begin{nopropn}
Suppose the root system of $G$ is of type $A_n$, $p\nmid |\pi_1(G)_\tor|$ and $\mm$ is a simple subquotient of a ``principal series $\cH$-module'' (that is, a simple subquotient of the $\cH$-module $\textnormal{Ind}_B^G(\chi)^{I(1)}$, where $B$ is a Borel subgroup of $G$ and $\chi$ is a smooth character of $B$).  Then $\mm$ has finite projective dimension over $\cH$ which is bounded above by the rank of $G$.  
\end{nopropn}

Our next goal will be to generalize the above proposition to an arbitrary simple module $\mm$ and an arbitrary $G$.  In Section \ref{ssingmods}, we recall the classification of simple supersingular right $\cH$-modules due to Ollivier and Vign\'eras (\cite{ollivier:compat} and \cite{vigneras:hecke3}).  Note that there is no such classification for supersingular \emph{representations} of $G$.  Using several reductions, we completely determine when such a module has finite projective dimension in Theorem \ref{ssing}.  This result will be required as input into our main theorem.

By Abe's classification (cf. \cite{abe:heckemods}), every simple right $\cH$-module is a subquotient of the parabolic coinduction of a simple supersingular $\cH_M$-module, where $M$ is a Levi subgroup of $G$ and $\cH_M$ is the associated pro-$p$-Iwahori--Hecke algebra.  Thus, our next task in Section \ref{sectionparabind} is to understand how projective dimension behaves under parabolic coinduction and taking subquotients.  The parabolic coinduction functor has an exact left adjoint (cf. \cite{vigneras:hecke5}), and a straightforward argument with Ext spaces shows that finiteness of projective dimension is preserved under parabolic coinduction (Lemma \ref{parabind}).  Passing to subquotients is less straightforward, and occupies the remainder of the section.

With all the pieces in place, we obtain our main result:

\begin{nothm}
Suppose $p\nmid|\pi_1(G)_\tor|$, and let $M$ denote a standard Levi subgroup of $G$.  Let $\nn$ be a simple supersingular right $\cH_M$-module, and $\mm$ a simple subquotient of the parabolic coinduction of $\nn$ from $\cH_M$ to $\cH$.  Then the following are equivalent:
\begin{itemize}
\item $\mm$ has finite projective dimension over $\cH$;
\item $\nn$ has finite projective dimension over $\cH_M$; 
\item the root system of $M$ is of type $A_1\times \cdots \times A_1$, and the characteristic function of $(I(1)\cap M)\alpha^\vee(x)(I(1)\cap M)$ acts trivially on $\nn$ for all $x$ in the residue field of $F$ and all simple roots $\alpha$ of $M$.  
\end{itemize}
Moreover, when $G$ is semisimple and $\mm$ satisfies the above conditions, the resolution of Ollivier--Schneider is a projective resolution of $\mm$, and the projective dimension of $\mm$ is equal to the rank of $G$.  
\end{nothm}

\noindent (In fact, we can strengthen the final statement somewhat; see below.)

This theorem shows that ``most'' simple $\cH$-modules have infinite projective dimension, and in particular, that simple supersingular modules are generically of infinite projective dimension.  On the other hand, simple subquotients of ``principal series $\cH$-modules'' have finite projective dimension.

In the final section (Section \ref{comp}), we present some complementary results.  First, we show how the above theorem adapts to simple modules over the \emph{Iwahori--Hecke algebra} $\cH'$, defined with respect to an Iwahori subgroup containing $I(1)$.  More precisely, the analog of the result above goes through mostly unchanged, except that the last condition is replaced by the simpler condition ``the root system of $M$ is of type $A_1\times \cdots \times A_1$'' (see Subsection \ref{iwahori} for more details).

Next, we specialize in Subsection \ref{represns} to the case $G = \textnormal{PGL}_2(\qp)$ or $\textnormal{SL}_2(\qp)$ with $p > 2$.  Recall that in this case, we have an equivalence between the category of $\cH$-modules and the category of smooth $G$-representations generated by their $I(1)$-invariant vectors.  By using this equivalence of categories, we demonstrate how to construct projective resolutions for certain irreducible smooth $G$-representations in the aforementioned representation category.  However, these resolutions will not be projective in the entire category of smooth $G$-representations.

Finally, we examine in Subsection \ref{zcharsection} what can be said when an $\cH$-module has a central character.  We prove that if such a module has finite projective dimension, the resolution of Ollivier--Schneider actually gives a projective resolution in the full subcategory of $\cH$-modules with the given central character.  This implies, in particular, the following fact: if $p\nmid|\pi_1(G/Z)|$, where $Z$ denotes the connected center of $G$, and $\mm$ is a simple $\cH$-module which has finite projective dimension over $\cH$, then the projective dimension of $\mm$ is in fact equal to the rank of $G$ (without any semisimplicity hypotheses).

\bigskip

\noindent \textbf{Acknowledgements}.  I would like to thank Noriyuki Abe, Rachel Ollivier, and Marie-France Vign\'eras for several useful conversations and feedback.  I would also like to thank the anonymous referee for helpful comments.  During the preparation of this article, funding was provided by NSF grant DMS-1400779 and an EPDI fellowship.

\section{Notation}\label{notation}

\subsection{General notation}
Let $p, F$ and $G$ be as in the introduction, so that $G$ is a split connected reductive group over a fixed nonarchimedean local field $F$ of residual characteristic $p$.  We let $k_F$ denote the residue field of $F$, and $q$ its order.

We will abuse notation throughout and conflate algebraic groups with their groups of $F$-points.  Denote by $Z$ the connected center of $G$, and let $r_{\textnormal{ss}}$ and $r_Z$ denote the semisimple rank of $G$ and the rank of $Z$, respectively.  Fix a maximal split torus $T$ of $G$ and let 
$$\langle-,-\rangle: X^*(T)\times X_*(T)\longrightarrow \bbZ$$ 
denote the natural perfect pairing.  The group $T$ acts on the standard apartment $(X_*(T)/X_*(Z))\otimes_\bbZ \bbR$ of the semisimple Bruhat--Tits building of $G$ by translation via $\nu$, where $\nu:T\longrightarrow X_*(T)$ is the homomorphism defined by
$$\langle\chi,\nu(\lambda)\rangle = -\textnormal{val}(\chi(\lambda))$$
for all $\chi\in X^*(T)$ and $\lambda\in T$, and where $\textnormal{val}:F^\times\longrightarrow \bbZ$ is the normalized valuation.

In the standard apartment, we fix a chamber $C$ and a hyperspecial vertex $x$ such that $x\in \overline{C}$.  Given a facet $\cF$ in the semisimple Bruhat--Tits building, we let $\cP_{\cF}$ denote the parahoric subgroup associated to $\cF$, and let $\cP_{\cF}^\dagger$ denote the stabilizer of $\cF$ in $G$; we have $\cP_\cF\subset \cP_{\cF}^\dagger$.  We set $I := \cP_C$, an Iwahori subgroup, and let $I(1)$ denote the pro-$p$-Sylow subgroup of $I$.

We identify the root system $\Phi\subset X^*(T)$ with the set of affine roots which are 0 on $x$.  Denote by $\Phi^+\subset \Phi$ the subset of affine roots which are positive on $C$; we have $\Phi = \Phi^+ \sqcup -\Phi^+$.  We let $\Pi$ denote the basis of $\Phi$ defined by $\Phi^+$, and define $B = T\ltimes U$ to be the Borel subgroup containing $T$ defined by $\Phi^+$, where $U$ is the unipotent radical of $B$.   A \emph{standard parabolic subgroup} $P = M\ltimes N$ is any parabolic subgroup containing $B$.  It will be tacitly assumed that all Levi subgroups $M$ appearing are standard; that is, they are Levi factors of standard parabolic subgroups and contain $T$.

Given a standard Levi subgroup $M$, we let $\Pi_M$ (resp. $\Phi_M$, resp. $\Phi_M^+$) denote the simple roots (resp. root system, resp. positive roots) defined by $M$.  Reciprocally, given a subset $J\subset \Pi$, we let $M_J$ denote the standard Levi subgroup it defines.

\subsection{Weyl groups}
Denote by $W_0 := N_G(T)/T$ the Weyl group of $G$, with length function $\ell:W_0\longrightarrow \bbZ_{\geq 0}$ (defined with respect to $\Pi$).  For a standard Levi subgroup $M$, we let $W_{M,0}\subset W_0$ denote the corresponding Weyl group, and set 
$$W_0^M := \{w\in W_0: w(\Pi_M)\subset \Phi^+\}.$$  
Every element $w$ of $W_0$ can be written as $w = vu$ for unique $v\in W_0^M, u\in W_{M,0}$, satisfying $\ell(w) = \ell(v) + \ell(u)$ (see \cite[Ch. IV, Exercices du $\S$1, (3)]{bourbaki:lie}).

We set 
$$\Lambda := T/(T\cap I),\qquad \widetilde{\Lambda} := T/(T\cap I(1)),$$
$$W := N_G(T)/(T\cap I),\qquad\tW := N_G(T)/(T\cap I(1)).$$
Note that the map $\nu$ descends to $\Lambda$ and $\widetilde{\Lambda}$.  For any standard Levi subgroup $M$, we let $W_M$ denote the subgroup of $W$ generated by $W_{M,0}$ and $\Lambda$ (that is, $W_M$ is the preimage of $W_{M,0}$ under the surjection $W\longtwoheadrightarrow W_0$).

The group $(T\cap I)/(T\cap I(1))$ identifies with the group of $k_F$-points of $T$, and we denote this group by $T(k_F)$.  Given any subset $X$ of $W$, we let $\widetilde{X}$ denote its preimage in $\tW$ under the natural projection $\tW\longtwoheadrightarrow W$, so that $\widetilde{X}$ is an extension of $X$ by $T(k_F)$.  For typographical reasons we write $\widetilde{X}_\square$ as opposed to $\widetilde{X_\square}$ if the symbol $X$ has some decoration $\square$.  We will usually denote generic elements of $\tW$ by $\tw$, and given an element $w\in W$ we often let $\widehat{w}\in \tW$ denote a specified choice of lift.

Since $x$ is hyperspecial we have $W_0 \cong (N_G(T)\cap \cP_x)/(T\cap \cP_x)$, which gives a section to the surjection $W\longtwoheadrightarrow W_0$.  We will always view $W_0$ as a a subgroup of $W$ via this section.  This gives the decomposition
$$W\cong W_0\ltimes \Lambda.$$
In particular, we see that any $\tw\in \tW$ may be written as $\widehat{\overline{w}}\lambda$, where $\overline{w}$ is the image of $\tw$ in $W_0\subset W$, $\widehat{\overline{w}}\in \tW$ is a fixed choice of lift, and $\lambda\in \widetilde{\Lambda}$.  Moreover, the length function $\ell$ on $W_0$ extends to $W$ and $\tW$ (see \cite[Cor. 5.10]{vigneras:hecke1}).

We also have a decomposition 
$$W\cong W_\aff\rtimes\Omega.$$ 
Here, $W_\aff$ is the affine Weyl group, generated by the set $S$ of simple affine reflections fixing the walls of $C$ (chosen as in \cite[\S 4.3]{ollivierschneider}).  Every element of $S$ is of the form $s_\alpha$, where $s_\alpha$ is the reflection in the hyperplane defined by the kernel of an affine root $\alpha$.  Moreover, the pair $(W_\aff,S)$ is a Coxeter system, and the restriction of $\ell$ to $W_\aff$ agrees with the length function of $W_\aff$ as a Coxeter group.  The group $\Omega$ is the subgroup of elements stabilizing $C$; equivalently, $\Omega$ is the subgroup of length 0 elements of $W$.  It is a finitely generated abelian group, and we write 
$$\Omega \cong \Omega_{\tor}\times\Omega_{\free}$$ 
where $\Omega_{\tor}$ is the (finite) torsion subgroup and $\Omega_{\free}$ is the free part.  Since the group $Z/(Z\cap I)$ embeds as a finite-index subgroup of $\Omega_{\free}$, we have $\Omega_{\free}\cong \bbZ^{\oplus r_Z}$.

\subsection{Hecke algebras}
Let $\cH$ denote the pro-$p$-Iwahori--Hecke algebra of $G$ with respect to $I(1)$ over $\fpb$:
$$\cH:= \End_{G}\left(\cind_{I(1)}^G(\bld)\right),$$
where $\bld$ denotes the trivial $I(1)$-module over $\fpb$ (see \cite{vigneras:hecke1} for details).  For any standard Levi subgroup $M$, let $\cH_M$ denote the analogously defined pro-$p$-Iwahori--Hecke algebra of $M$ with respect to $I_M(1):=I(1)\cap M$ (which is \emph{not} a subalgebra of $\cH$ in general).  For any facet $\cF\subset \overline{C}$, we let 
$$\cH_{\cF} := \End_{\cP_{\cF}}\left(\cind_{I(1)}^{\cP_{\cF}}(\bld)\right),\qquad \cH_{\cF}^\dagger := \End_{\cP_{\cF}^\dagger}\left(\cind_{I(1)}^{\cP_{\cF}^\dagger}(\bld)\right);$$
extending functions on $\cP_\cF$ by zero to $G$ gives a $\cP_\cF$-equivariant injection 
$$\cind_{I(1)}^{\cP_{\cF}}(\bld)\longhookrightarrow \cind_{I(1)}^{G}(\bld),$$ 
which induces 
$$\cH_{\cF} = \End_{\cP_\cF}\left(\cind_{I(1)}^{\cP_{\cF}}(\bld)\right) \longhookrightarrow \Hom_{\cP_\cF}\left(\cind_{I(1)}^{\cP_{\cF}}(\bld),\cind_{I(1)}^{G}(\bld)\right) \cong \End_G\left(\cind_{I(1)}^{G}(\bld)\right) = \cH$$
(and similarly for $\cH_{\cF}^\dagger$).  We therefore view $\cH_{\cF}$ and $\cH_{\cF}^\dagger$ as subalgebras of $\cH$, and $\cH_{\cF}$ as a subalgebra of $\cH_{\cF}^\dagger$ (see \cite[\S\S 3.3.1 and 4.9]{ollivierschneider} for more details).

We view $\cH_M$ as the convolution algebra of $\fpb$-valued, $I_M(1)$-biinvariant functions on $M$.  The group $\tW_M$ gives a full set of coset representatives for $I_M(1)\backslash M/I_M(1)$, and for $\tw\in \tW_M$, we let $\T_{\tw}^M$ denote the characteristic function of $I_M(1)\tw I_M(1)$ (and drop the superscript when $M = G$).  For standard properties of the algebras $\cH_M$ (quadratic relations, Bernstein basis, definition of the elements $\T_{\tw}^{M,*}$, etc.), we defer to \cite{vigneras:hecke1}.  Let us only recall the braid relations: if $\tw,\tw'\in \tW_M$ satisfy $\ell_M(\tw\tw') = \ell_M(\tw) + \ell_M(\tw')$, where $\ell_M$ is the length function on $\tW_M$, then
$$\T_{\tw}^M\T_{\tw'}^M = \T_{\tw\tw'}^M.$$

\section{Preliminary results}\label{prelim}

We first record some simple results concerning the algebras $\cH_{\cF}$ and $\cH_{\cF}^\dagger$.  Given a right module $\mm$ over an associative unital ring $\cR$, we let $\pd_{\cR}(\mm)$ and $\id_{\cR}(\mm)$ denote the projective and injective dimensions of $\mm$ over $\cR$, respectively.

\begin{remark}\label{projinjdim}
Let $n\in \bbZ_{\geq	 0}$.  Recall that an associative unital (left and right) noetherian ring $\cR$ is called \emph{$n$-Gorenstein} if $\id_{\cR}(\cR) \leq n$, where $\cR$ is viewed as either a left or right $\cR$-module (\cite[Def. 9.1.9]{enochsjenda}).  Fix a facet $\cF\subset \overline{C}$.  By \cite[Thm. 3.14, Prop. 5.5, Remarks following Lem. 5.2]{ollivierschneider} the algebras $\cH_{\cF}$, $\cH_{\cF}^\dagger$, and $\cH$ are all $n$-Gorenstein, where $n = 0, r_Z,$ and $r_{\textnormal{ss}} + r_Z$, respectively.  We shall use the following fact several times (\cite[Thm. 9.1.10]{enochsjenda}): let $\cR\in \{\cH_{\cF}, \cH_{\cF}^\dagger, \cH\}$, so that $\cR$ is $n$-Gorenstein with $n$ as above, and let $\mm$ be a right module over $\cR$.  Then the following are equivalent:
\begin{itemize}
\item $\id_{\cR}(\mm) < \infty$;
\item $\pd_{\cR}(\mm) < \infty$;
\item $\id_{\cR}(\mm) \leq n$;
\item $\pd_{\cR}(\mm) \leq n$.
\end{itemize}
\end{remark}

\begin{remark}
We will employ the Eckmann--Shapiro lemma extensively, so we briefly recall it here (see \cite[Cor. 2.8.4]{benson:1} for more details).  Let $\cA$ and $\cB$ be two unital rings.  Suppose $\cA$ is a subring of $\cB$, and that $\cB$ is projective as a left (resp. right) $\cA$-module.  If $\mm$ is a right $\cB$-module and $\nn$ is a right $\cA$-module, we then have:
\begin{eqnarray*}
\bullet~ \Ext_{\cA}^i\left(\nn,\mm|_{\cA}\right) & \cong & \Ext_{\cB}^i\left(\nn\otimes_{\cA}\cB,\mm\right);\\
\bullet~ \Ext_{\cA}^i\left(\mm|_{\cA},\nn\right) & \cong & \Ext_{\cB}^i\left(\mm,\Hom_{\cA}(\cB,\nn)\right).
\end{eqnarray*}
\end{remark}

\begin{lemma}\label{restr}
Let $\cF, \cF'$ denote two facets such that $\cF'\subset\overline{\cF}$, $\cF\subset \overline{C}$, and let $\mm$ be a right $\cH_{\cF'}$-module.  If $\mm$ is projective, then the $\cH_{\cF}$-module $\mm|_{\cH_{\cF}}$ is projective.
\end{lemma}

\begin{proof}
Given $\cF'\subset\overline{\cF}$, we let $W_\cF\subset W_{\cF'}$ denote the subgroups of $W_\aff$ generated by those elements of $S$ which fix pointwise each respective facet (see \cite[\S 4.3]{ollivierschneider} for more details).  The group $W_{\cF}$ identifies with a parabolic subgroup of $W_{\cF'}$, and we let $W^{\cF}$ denote the set of minimal coset representatives of $W_{\cF'}/W_{\cF}$ (cf. definition of $W_0^M$).  For each $v\in  W^\cF\subset W_{\cF'}$, we fix a lift $\widehat{v}\in \tW_{\cF'}$.  Then any element $\tw\in \tW_{\cF'}$ may be written uniquely as $\tw = \widehat{v}\tu$, with $v\in W^{\cF}, \tu\in \tW_\cF$, satisfying $\ell(\tw) = \ell(\widehat{v}) + \ell(\tu)$.

By the comments preceding \cite[Lem. 4.20]{ollivierschneider}, the algebra $\cH_{\cF}$ has a basis given by $\{\T_{\tw}\}_{\tw\in \tW_{\cF}}$ (and similarly for $\cH_{\cF'}$).  Using the above remarks, we see that $\cH_{\cF}$ identifies with a parabolic subalgebra of $\cH_{\cF'}$, and $\cH_{\cF'}$ is free as a right $\cH_{\cF}$-module, with basis $\{\T_{\widehat{v}}\}_{v\in W^{\cF}}$.  Therefore, by the Eckmann--Shapiro Lemma, 
$$\Ext_{\cH_{\cF}}^i\left(\mm|_{\cH_{\cF}},\nn\right) \cong \Ext_{\cH_{\cF'}}^i\left(\mm,\Hom_{\cH_{\cF}}(\cH_{\cF'},\nn)\right) = 0~\textnormal{for}~i > 0.$$
\end{proof}

\begin{lemma}\label{daggerproj}
Assume that $p\nmid|\Omega_{\tor}|$.  Let $\cF$ denote a facet such that $\cF\subset \overline{C}$, and let $\mm$ be a right $\cH_{\cF}^\dagger$-module.  Then $\mm|_{\cH_{\cF}}$ is projective over $\cH_{\cF}$ if and only if $\pd_{\cH_{\cF}^\dagger}(\mm)\leq r_Z$.
\end{lemma}

\begin{proof}
Let $\Omega_{\cF}$ denote the subgroup of $\Omega$ stabilizing $\cF$, and write $\Omega_{\cF}\cong \Omega_{\cF,\tor}\times \Omega_{\cF,\free}$, where $\Omega_{\cF,\tor}$ is finite and $\Omega_{\cF,\free}$ is free of rank $r_Z$ (note that $Z/(Z\cap I)$ embeds as a finite-index subgroup of $\Omega_{\cF,\free}$).  The algebra $\cH_{\cF}^\dagger$ is generated by $\cH_{\cF}$ and $\{\T_{\widetilde{\omega}}\}_{\widetilde{\omega}\in \tOmega_{\cF}}$ (\cite[Lem. 4.20]{ollivierschneider}), and we let $\cH_{\cF,\free}$ denote the subalgebra of $\cH_{\cF}^\dagger$ generated by $\cH_{\cF}$ and $\{\T_{\widetilde{\omega}}\}_{\widetilde{\omega}\in \tOmega_{\cF,\free}}$.  We fix a set of generators $\{\omega_i\}_{i = 1}^{r_Z}$ for $\Omega_{\cF,\free}$, and let $\{\widehat{\omega}_i\}_{i = 1}^{r_Z}\subset \tOmega_{\cF,\free}$ denote a fixed set of lifts.  Using the braid relations, we see that $\cH_{\cF,\free}$ is free over $\cH_{\cF}$ with basis $\{\T_{\widehat{\omega}_1^{\ell_1}\cdots\widehat{\omega}_{r_Z}^{\ell_{r_Z}}}\}_{\ell_i\in \bbZ}$.  Moreover, this gives $\cH_{\cF,\free}$ the structure of a(n iterated) skew Laurent polynomial algebra over $\cH_{\cF}$ (see \cite[\S\S 1.2 and 1.4.3]{mcconnellrobson} for the relevant definition).  Therefore, if $\mm|_{\cH_{\cF}}$ is projective, \cite[Prop. 7.5.2(ii)]{mcconnellrobson} gives
$$\pd_{\cH_{\cF,\free}}(\mm|_{\cH_{\cF,\free}})\leq \pd_{\cH_{\cF}}(\mm|_{\cH_{\cF}}) + r_Z = r_Z.$$

Now, fix a set of lifts $\{\widehat{\omega}\}_{\omega\in \Omega_{\cF,\tor}} \subset \tOmega_{\cF,\tor}$ containing 1.  Once again using the braid relations, we see that $\cH_{\cF}^\dagger$ is free over $\cH_{\cF,\free}$, with basis given by the elements $\{\T_{\widehat{\omega}}\}_{\omega\in \Omega_{\cF,\tor}}$.  This gives $\cH_{\cF}^\dagger$ the structure of a crossed product algebra: $\cH_{\cF}^\dagger \cong \cH_{\cF,\free}*\Omega_{\cF,\tor}$ (see \cite[\S 1.5.8]{mcconnellrobson}).  Since $p\nmid |\Omega_{\cF,\tor}|$ by assumption, \cite[Thm. 7.5.6(ii)]{mcconnellrobson} implies
$$\pd_{\cH_{\cF}^\dagger}(\mm) = \pd_{\cH_{\cF,\free}}(\mm|_{\cH_{\cF,\free}})\leq r_Z.$$

To prove the converse, recall that $\cH_{\cF}^\dagger$ is free over $\cH_{\cF}$, and therefore any projective resolution of $\mm$ restricts to a projective resolution of $\mm|_{\cH_{\cF}}$.  Hence, if $\pd_{\cH_{\cF}^\dagger}(\mm)\leq r_Z$, we then obtain $\pd_{\cH_{\cF}}(\mm|_{\cH_{\cF}})\leq r_Z < \infty$, and $\mm|_{\cH_{\cF}}$ must be projective by Remark \ref{projinjdim}.
\end{proof}

\begin{lemma}\label{twist}
Let $\cF$ denote a facet such that $\cF\subset \overline{C}$, let $\mm$ be a right $\cH_{\cF}^\dagger$-module, and let $n\in \bbZ_{\geq 0}$.  Then $\pd_{\cH_{\cF}^\dagger}(\mm)\leq n$ if and only if $\pd_{\cH_{\cF}^\dagger}(\mm(\epsilon_{\cF}))\leq n$, where $\epsilon_{\cF}$ denotes the orientation character of $\cP_{\cF}^\dagger$ (see \cite[\S\S 3.1 and 3.3.1]{ollivierschneider}).  
\end{lemma}

\begin{proof}
This follows from the fact that the functor $\mm\longmapsto \mm(\epsilon_{\cF})$ is exact on the category of $\cH_{\cF}^\dagger$-modules, and preserves projectives.  
\end{proof}

The following result will be used in a subsequent section.  We use notation and terminology from \cite[\S 4.1]{abe:heckemods}.  For a standard Levi subgroup $M$, we let $\cH_M^-$ denote the subalgebra of $\cH_M$ consisting of functions supported on $M$-negative elements.  Recall that, if $w\in W_{M,0}\subset W_M$ has a fixed lift $\widehat{w}\in \tW_M$ and $\lambda\in \widetilde{\Lambda}$, the element $\lambda\widehat{w}\in \tW_M$ is \emph{$M$-negative} if $\langle\alpha,\nu(\lambda)\rangle \geq 0$ for every $\alpha\in \Phi^+\smallsetminus\Phi_M^+$.  By Lemma 4.6 of \emph{loc. cit.}, we have an injective algebra morphism 
\begin{eqnarray*}
j_M^-:\cH_M^- & \longhookrightarrow & \cH\\
\T_{\tw}^{M,*} & \longmapsto & \T_{\tw}^*
\end{eqnarray*}
for $M$-negative $\tw\in \tW_M$, and we view $\cH$ as a right $\cH_M^-$-module via $j_M^-$.  Note that, while $\cH_M$ depends only on the Levi subgroup $M$, $\cH_M^-$ and $j_M^-$ depend on the choice of positive roots, and therefore on the choice of hyperspecial vertex $x\in \overline{C}$.  

\begin{lemma}[Mackey formula]\label{mackey}
Let $M$ be a standard Levi subgroup, and $\nn$ a right $\cH_M$-module.  We then have an isomorphism of right $\cH_x$-modules
\begin{eqnarray*}
\textnormal{Hom}_{\cH_M^-}(\cH,\nn)|_{\cH_x} & \stackrel{\sim}{\longrightarrow} & \textnormal{Hom}_{\cH_M^-\cap\cH_x}(\cH_x,\nn|_{\cH_M^-\cap \cH_x})\\
\varphi & \longmapsto & \varphi|_{\cH_x}.
\end{eqnarray*}
\end{lemma}

\begin{proof}
The restriction map is clearly well-defined and $\cH_x$-equivariant.  It therefore suffices to check that it is an isomorphism of vector spaces.

Recall that the choice of hyperspecial vertex $x$ gives an identification $W_0 \cong W_x$, where $W_x$ denotes the subgroup of $W_\aff$ generated by those elements of $S$ which fix $x$.  For each $v\in W_0^M$, we fix a lift $\widehat{v}\in\tW_0^M$.  Then, given any element $\tw\in \tW_0 \cong \tW_x$, we may write $\tw = \widehat{v}\tu$ for unique $v\in W_0^M, \tu\in \tW_{M,0}$, satisfying $\ell(\tw) = \ell(\widehat{v}) + \ell(\tu)$.  One sees easily that $\{\T_{\tu}\}_{\tu\in \tW_{M,0}}$ gives a basis for $\cH_M^- \cap \cH_x$ and $\{\T_{\tw}\}_{\tw\in \tW_0}$ gives a basis for $\cH_x$, so that $\cH_M^-\cap \cH_x$ identifies with a parabolic subalgebra of $\cH_x$.  Therefore, the braid relations and the factorization $\tw = \widehat{v}\tu$ imply that $\cH_x$ is free as a right $\cH_M^-\cap \cH_x$-module, with basis $\{\T_{\widehat{v}}\}_{v\in W_0^M}$, and the map
\begin{eqnarray*}
\Hom_{\cH_M^-\cap \cH_x}(\cH_x,\nn|_{\cH_M^-\cap \cH_x}) & \longrightarrow & \bigoplus_{v\in W_0^M} \nn\\
 \psi & \longmapsto & (\psi(\T_{\widehat{v}}))_{v\in W_0^M}
\end{eqnarray*}
is an isomorphism of vector spaces.

Consider now the diagram of vector spaces
\begin{center}
\begin{tikzcd}
\textnormal{Hom}_{\cH_M^-}(\cH,\nn)|_{\cH_x} \ar[r]  \ar[d ]& \textnormal{Hom}_{\cH_M^-\cap\cH_x}(\cH_x,\nn|_{\cH_M^-\cap \cH_x}) \ar[d,"{\rotatebox{90}{$\sim$}}"]\\
\displaystyle{ \bigoplus_{v\in W_0^M}\nn }&  \displaystyle{ \bigoplus_{v\in W_0^M}\nn }
\end{tikzcd}
\end{center}
where the right vertical map is the map of the previous paragraph, and the left vertical map is given by $\varphi \longmapsto (\varphi(\T_{\widehat{v}}))_{v\in W_0^M}$.  By \cite[Lem. 4.10]{abe:heckemods}, the left vertical map is also an isomorphism of vector spaces, and composing its inverse with the horizontal restriction map and the right vertical map gives the identity on $\bigoplus_{v\in W_0^M}\nn$.  Therefore the restriction map is an isomorphism.  
\end{proof}

\section{Resolutions and spectral sequences}\label{resnsandsss}

We now consider resolutions and projective dimensions of $\cH$-modules.  For $0\leq i \leq r_{\textnormal{ss}}$, we fix a finite set $\mathscr{F}_i$ of representatives of the $G$-orbits of $i$-dimensional facets in the Bruhat--Tits building of $G$, subject to the condition that every element of $\mathscr{F}_i$ is contained in $\overline{C}$.  Given a right $\cH$-module $\mm$, \cite[Thm. 3.12]{ollivierschneider} implies that we have a resolution
\begin{equation}\label{resn}
0\longrightarrow \bigoplus_{\cF\in \mathscr{F}_{r_{\textnormal{ss}}}}\mm|_{\cH_{\cF}^\dagger}(\epsilon_{\cF})\otimes_{\cH_{\cF}^\dagger}\cH\longrightarrow \ldots \longrightarrow \bigoplus_{\cF\in \mathscr{F}_0}\mm|_{\cH_{\cF}^\dagger}(\epsilon_{\cF})\otimes_{\cH_{\cF}^\dagger}\cH \longrightarrow \mm \longrightarrow 0,
\end{equation}
which gives a hyper-Ext spectral sequence
\begin{equation}\label{ss}
E_1^{i,j} = \Ext_{\cH}^j\left(\bigoplus_{\cF\in \mathscr{F}_i}\mm|_{\cH_{\cF}^\dagger}(\epsilon_{\cF})\otimes_{\cH_{\cF}^\dagger}\cH,\nn\right)\cong\bigoplus_{\cF\in \mathscr{F}_i}\Ext_{\cH_{\cF}^\dagger}^j\left(\mm|_{\cH_{\cF}^\dagger}(\epsilon_{\cF}),\nn|_{\cH_{\cF}^\dagger}\right)\Longrightarrow \Ext^{i + j}_{\cH}(\mm,\nn).
\end{equation}
(The isomorphism on the left-hand side of the spectral sequence above follows from Proposition 4.21 of \emph{loc. cit.} and the Eckmann--Shapiro lemma.)

\begin{remark}
We have the following variation on the above.  Let $G_\aff$ denote the subgroup of $G$ generated by all parahoric subgroups, and let $\cH_\aff$ denote the subalgebra of $\cH$ consisting of elements with support in $G_\aff$.  We view the algebras $\cH_\cF$ as subalgebras of $\cH_\aff$.

The group $G_\aff$ acts on the Bruhat--Tits building of $G$, with a transitive action on chambers.  For $0\leq i \leq r_{\textnormal{ss}}$, we fix a finite set $\mathscr{F}_i^{\aff}$ of representatives of the $G_\aff$-orbits of $i$-dimensional facets, subject to the condition that every element of $\mathscr{F}_i^{\aff}$ is contained in $\overline{C}$.  Note that the sets $\mathscr{F}_i$ and $\mathscr{F}_i^{\aff}$ are different in general.  The results of \cite{ollivierschneider} easily modify to this setting, and given a right $\cH_\aff$-module $\mm$, we obtain a resolution
\begin{equation}\label{affresn}
0\longrightarrow \bigoplus_{\cF\in \mathscr{F}_{r_{\textnormal{ss}}}^{\aff}}\mm|_{\cH_{\cF}}\otimes_{\cH_{\cF}}\cH_{\aff}\longrightarrow \ldots \longrightarrow \bigoplus_{\cF\in \mathscr{F}_0^{\aff}}\mm|_{\cH_{\cF}}\otimes_{\cH_{\cF}}\cH_{\aff} \longrightarrow \mm \longrightarrow 0,
\end{equation}
which gives a spectral sequence
\begin{equation}\label{affss}
E_1^{i,j} = \Ext_{\cH_{\aff}}^j\left(\bigoplus_{\cF\in \mathscr{F}_i^{\aff}}\mm|_{\cH_{\cF}}\otimes_{\cH_{\cF}}\cH_{\aff},\nn\right) \cong \bigoplus_{\cF\in \mathscr{F}_i^{\aff}}\Ext_{\cH_{\cF}}^j\left(\mm|_{\cH_{\cF}},\nn|_{\cH_{\cF}}\right)\Longrightarrow \Ext^{i + j}_{\cH_{\aff}}(\mm,\nn).
\end{equation}
In particular, the resolution \eqref{affresn} shows that $\cH_\aff$ is $r_{\textnormal{ss}}$-Gorenstein (cf. Propositions 1.2(i) and 4.21 of \emph{loc. cit.}).
\end{remark}

The following simple observation will be useful later.

\begin{lemma}\label{redtofacets}
Let $\mm$ denote a right $\cH$-module.  Then $\pd_{\cH}(\mm) < \infty$ if and only if $\pd_{\cH_{\cF}^\dagger}(\mm|_{\cH_{\cF}^\dagger}) < \infty$ for every facet $\cF\subset \overline{C}$.  Moreover, if $p\nmid|\Omega_{\tor}|$, then $\pd_{\cH}(\mm) < \infty$ if and only if $\mm|_{\cH_{\cF}}$ is projective over $\cH_{\cF}$ for every facet $\cF\subset \overline{C}$.  
\end{lemma}

The same claims obviously apply to $\cH_{\aff}$ and $\cH_{\cF}$ (without the need for the ``$p\nmid|\Omega_\tor|$'' assumption).  

\begin{proof}
\cite[Prop. 4.21]{ollivierschneider} shows that $\cH$ is free over $\cH_{\cF}^\dagger$ as either a left or right module.  The Eckmann--Shapiro lemma then gives
$$\Ext_{\cH_{\cF}^\dagger}^i\left(\mm|_{\cH_{\cF}^\dagger},\nn\right)\cong \Ext_{\cH}^i\left(\mm,\Hom_{\cH_{\cF}^\dagger}(\cH,\nn)\right),$$
which gives one implication.  On the other hand, if $\pd_{\cH_{\cF}^\dagger}(\mm|_{\cH_{\cF}^\dagger}) < \infty$ for all $\cF\subset \overline{C}$, then the $E_1^{i,j}$ term in \eqref{ss} vanishes for $i$ and $j$ sufficiently large (this uses Lemma \ref{twist}).  The spectral sequence \eqref{ss} then shows that $\Ext_{\cH}^i(\mm,\nn) = 0$ for all $\cH$-modules $\nn$ and $i$ sufficiently large (independent of $\nn$), so that $\pd_{\cH}(\mm) < \infty$.

If we assume $p\nmid|\Omega_{\tor}|$, then Remark \ref{projinjdim} and Lemma \ref{daggerproj} imply that $\pd_{\cH_{\cF}^\dagger}(\mm|_{\cH_{\cF}^\dagger}) < \infty$ if and only if $\mm|_{\cH_{\cF}}$ is projective over $\cH_{\cF}$.  
\end{proof}

\section{Simple modules}\label{simplemods}

We recall the classification of simple $\cH$-modules from \cite{abe:heckemods}.  Let $P = M\ltimes N$ denote a standard parabolic subgroup, and let $\nn$ denote a simple supersingular right $\cH_M$-module (see \cite[Def. 6.10]{vigneras:hecke3} and Section \ref{ssingmods} below).  Set
$$\Pi(\nn) := \Pi_M \sqcup \left\{\alpha\in \Pi: \begin{array}{l} \diamond~\langle\beta,\alpha^\vee\rangle = 0 ~\textnormal{for all}~ \beta\in \Pi_M\\ \diamond~ \T^M_{\alpha^\vee(x)}~\textnormal{acts trivially on $\nn$ for all}~x\in F^\times\end{array}\right\},$$
and let $P(\nn) = M(\nn)\ltimes N(\nn)$ denote the corresponding standard parabolic subgroup.  Given any standard parabolic subgroup $Q = L\ltimes V$ satisfying $P\subset Q\subset P(\nn)$, \cite[Prop. 4.16]{abe:heckemods} shows that $\nn$ extends uniquely to an $\cH_L$-module satisfying certain properties.  We denote this extension by $\nn^{e_L}$.  Given any right $\cH_L$-module $\vv$, we define the parabolic coinduction
$$I_L(\vv) := \Hom_{\cH_L^-}(\cH, \vv),$$
where $\cH$ is viewed as a right $\cH_L^-$-module via $j_L^-$, and set
$$I(P,\nn,Q) := I_L(\nn^{e_L})\Bigg/\left(\sum_{Q \subsetneq Q' \subset P(\nn)} I_{L'}(\nn^{e_{L'}})\right).$$
By \cite[Thm. 4.22]{abe:heckemods}, the $\cH$-modules $I(P,\nn,Q)$ are simple and $I(P,\nn,Q)\cong I(P',\nn',Q')$ if and only if $P = P', Q = Q',$ and $\nn\cong \nn'$.  Moreover, the $I(P,\nn,Q)$ exhaust all isomorphism classes of simple right $\cH$-modules (for varying $P, \nn$, and $Q$).  Using \cite[Cor. 4.26]{abe:heckemods} and the fact that $I_{M}(\nn)$ is multiplicity-free, we get that the Jordan--H\"older factors of $I_L(\nn^{e_L})$ are given by
$$\{I(P,\nn, Q'): Q\subset Q' \subset P(\nn)\}.$$

We now give a result on the structure of $I(P,\nn,Q)$.  In the proof below, if $J$ is a subset of $\Pi$ such that $\Pi_M\subset J \subset \Pi(\nn)$, we use $I_J(\nn)$ to denote $I_{M_J}(\nn^{e_{M_J}})$.

\begin{lemma}\label{cech}
The \v{C}ech complex $\cC_\bullet$ of $I(P,\nn,Q)$
\begin{flushleft}
$\displaystyle{0\longrightarrow I_{M(\nn)}(\nn^{e_{M(\nn)}}) \stackrel{\partial_r}{\longrightarrow} \bigoplus_{\sub{L\subset L'\subset M(\nn)}{|\Pi_{L'}\smallsetminus\Pi_L| = r - 1}}I_{L'}(\nn^{e_{L'}})\stackrel{\partial_{r - 1}}{\longrightarrow} \ldots }$
\end{flushleft}
\begin{flushright}
$\displaystyle{\ldots \stackrel{\partial_{2}}{\longrightarrow} \bigoplus_{\sub{L\subset L'\subset M(\nn)}{|\Pi_{L'}\smallsetminus\Pi_L| = 1}}I_{L'}(\nn^{e_{L'}})\stackrel{\partial_1}{\longrightarrow}I_L(\nn^{e_L})\stackrel{\partial_0}{\longrightarrow} I(P,\nn,Q)\longrightarrow 0}$
\end{flushright}
is exact, where $r = |\Pi(\nn)\smallsetminus\Pi_L|$.  
\end{lemma}

We will abbreviate the resolution above by $\cC_\bullet \longrightarrow I(P,\nn,Q)\longrightarrow 0$.

\begin{proof}
The differentials are given as follows.  We fix a numbering $\Pi(\nn)\smallsetminus \Pi_L = \{\alpha_1, \ldots, \alpha_r\}$, and let $(f_{J'})_{J'}\in \bigoplus_{\sub{\Pi_L\subset J'\subset \Pi(\nn)}{|J'\smallsetminus \Pi_L| = j}} I_{J'}(\nn) = \cC_j$ for $1\leq j \leq r$.  Let $J''$ denote a subset of $\Pi$ such that $\Pi_L \subset J'' \subset \Pi(\nn)$ and $|J''\smallsetminus \Pi_L| = j - 1$, and write $\Pi(\nn)\smallsetminus J'' = \{\alpha_{\ell_1}, \ldots, \alpha_{\ell_{r - j + 1}}\}$ with $\ell_1 < \ldots < \ell_{r - j + 1}$.  We then have
$$(\textnormal{pr}_{J''}\circ\partial_j)((f_{J'})_{J'}) = \sum_{k = 1}^{r - j + 1}(-1)^k f_{J''\cup\{\alpha_{\ell_k}\}}.$$
The standard argument shows that this gives a complex.

Exactness at $\cC_0$ is clear.  Suppose now that $(f_{J'})_{J'} \in \ker(\partial_j)$ for $1\leq j \leq r$, and set $J_0' := \Pi_L\cup\{\alpha_1,\ldots, \alpha_j\}$.  We have
$$(\textnormal{pr}_{J_0'\smallsetminus\{\alpha_j\}}\circ\partial_j)((f_{J'})_{J'}) =  \sum_{k = 1}^{r - j + 1}(-1)^k f_{(J_0'\smallsetminus\{\alpha_j\})\cup \{\alpha_{j - 1 + k}\}} = 0,$$
which implies $f_{J_0'}\in I_{J_0'}(\nn)\cap \left(\sum_{k = 2}^{r - j + 1}I_{(J_0'\smallsetminus\{\alpha_j\})\cup \{\alpha_{j - 1 + k}\}}(\nn)\right)$.  We clearly have
\begin{eqnarray*}
\sum_{k = 2}^{r - j + 1} I_{J_0' \cup\{\alpha_{j - 1 + k}\}}(\nn) & \subset & \sum_{k = 2}^{r - j + 1} \left(I_{J_0'}(\nn)\cap I_{ (J_0'\smallsetminus\{\alpha_j\})\cup\{\alpha_{j - 1 + k}\}}(\nn)\right)\\
 & \subset & I_{J_0'}(\nn)\cap \left(\sum_{k = 2}^{r - j + 1}I_{(J_0'\smallsetminus\{\alpha_j\})\cup \{\alpha_{j - 1 + k}\}}(\nn)\right),
\end{eqnarray*}
and comparing the Jordan--H\"older factors of both sides shows both inclusions must be equalities (this uses the fact that $I_M(\nn)$ is multiplicity-free).  
Write 
$$f_{J_0'} = \sum_{k = 2}^{r - j + 1}(-1)^{k - 1} g_{J_0'\cup \{\alpha_{j - 1 + k}\}},$$
 with $g_{J_0'\cup \{\alpha_{j - 1 + k}\}}\in I_{J_0'\cup\{\alpha_{j - 1 + k}\}}(\nn)$, and define $(g_K)_K\in \bigoplus_{\sub{\Pi_L\subset K\subset \Pi(\nn)}{|K\smallsetminus\Pi_L| = j + 1}}I_K(\nn) = \cC_{j + 1}$ by 
$$g_K = \begin{cases}g_{J_0'\cup\{\alpha_{j - 1 + k}\}} & \textnormal{if}~K = J_0'\cup\{\alpha_{j - 1 + k}\}, k\geq 2,\\ 0 & \textnormal{otherwise.}\end{cases}$$
This gives 
$$\textnormal{pr}_{J_0'}((f_{J'})_{J'} - \partial_{j + 1}((g_K)_K)) = f_{J_0'} - \sum_{k = 1}^{r - j}(-1)^k g_{J_0'\cup\{\alpha_{j + k}\}} = 0.$$
Replacing $(f_{J'})_{J'}$ by $(f_{J'})_{J'} - \partial_{j + 1}((g_K)_K)$, we may assume $f_{\Pi_L\cup\{\alpha_1,\ldots, \alpha_j\}} = 0$.

We now fix $j \leq s \leq r - 1$, and assume that $f_{J'} = 0$ for every $J'$ with $J'\smallsetminus \Pi_L\subset \{\alpha_1,\ldots, \alpha_s\}$.  Let $J_0'$ be such that $J_0'\smallsetminus\Pi_L\subset \{\alpha_1,\ldots,\alpha_{s + 1}\}$ and $\alpha_{s + 1}\in J_0'$, and write $\Pi(\nn)\smallsetminus J_0' = \{\alpha_{\ell_1}, \ldots, \alpha_{\ell_{r - j}}\}$ as above.   By assumption, we have
$$(\textnormal{pr}_{J_0'\smallsetminus\{\alpha_{s + 1}\}}\circ\partial_j)((f_{J'})_{J'}) = \pm f_{J_0'} + \sum_{\sub{k = 1}{\ell_k> s + 1}}^{r - j}(-1)^{k + 1}f_{(J_0'\smallsetminus\{\alpha_{s + 1}\})\cup\{\alpha_{\ell_k}\}} = 0,$$
and as above we may write 
$$f_{J_0'} = \sum_{\sub{k = 1}{\ell_k> s + 1}}^{r - j} (-1)^{k}g_{J_0'\cup\{\alpha_{\ell_k}\}}$$ 
with $g_{J_0'\cup\{\alpha_{\ell_k}\}}\in I_{J_0'\cup\{\alpha_{\ell_k}\}}(\nn)$.  We define $(g_K)_K\in \bigoplus_{\sub{\Pi_L\subset K\subset \Pi(\nn)}{|K\smallsetminus\Pi_L| = j + 1}}I_K(\nn) = \cC_{j + 1}$ by
$$g_K = \begin{cases}g_{J_0'\cup\{\alpha_{\ell_k}\}} & \textnormal{if}~K = J_0'\cup\{\alpha_{\ell_k}\}, \ell_k > s + 1,\\ 0 & \textnormal{otherwise.}\end{cases}$$

Fix $J_1'$ with $J_1'\smallsetminus \Pi_L \subset \{\alpha_1,\ldots, \alpha_{s + 1}\}$ and $|J_1'\smallsetminus\Pi_L| = j$.  One easily checks that
$$\textnormal{pr}_{J_1'}((f_{J'})_{J'} - \partial_{j + 1}((g_K)_K)) = \begin{cases}0 & \textnormal{if}~J_1' = J_0',\\ f_{J_1'} & \textnormal{if}~J_1' \neq J_0'.\end{cases}$$
Therefore, replacing $(f_{J'})_{J'}$ by $(f_{J'})_{J'} - \partial_{j + 1}((g_K)_K)$ (several times if necessary), we may assume that $f_{J'} = 0$ for every $J'$ with $J'\smallsetminus \Pi_L \subset \{\alpha_1,\ldots, \alpha_{s + 1}\}$.  By induction on $s$ we get that $(f_{J'})_{J'}\in \textnormal{im}(\partial_{j + 1})$.  
\end{proof}

\begin{remark}
The above proof remains valid for an arbitrary (i.e., not necessarily split) connected reductive group.  
\end{remark}

Now let $\psi:T\longrightarrow \fpb^\times$ be a smooth character of $T$.  We use the same notation $\psi$ to denote the corresponding right $\cH_T$-module, with action given by $v\cdot \T_\lambda^T = \psi(\lambda)^{-1}v$, for $v\in \psi, \lambda\in \widetilde{\Lambda}$.  Note that the $\cH_T$-module $\psi$ is supersingular (see Section \ref{ssingmods}).  We will call $I_T(\psi)$ a \emph{principal series module}.  In this setting, we have $\Pi(\psi) = \{\alpha\in \Pi:\psi\circ\alpha^\vee(x) = 1~\textnormal{for all}~x\in F^\times\}$, and we let $P(\psi)$ denote the corresponding standard parabolic subgroup.

\begin{lemma}\label{psproj} 
Let $\psi:T\longrightarrow \fpb^\times$ be a smooth character, and let $Q = L\ltimes V$ denote a standard parabolic subgroup with $B\subset Q\subset P(\psi)$.  
\begin{enumerate}
\item The right $\cH_x$-module $I_L(\psi^{e_L})|_{\cH_{x}}$ is projective.  
\item The right $\cH_{x}$-module $I(B,\psi,Q)|_{\cH_{x}}$ is projective.
\end{enumerate}
\end{lemma}

\begin{proof}
(1) Using the Mackey formula (Lemma \ref{mackey}), we have
$$I_L(\psi^{e_L})|_{\cH_x} \cong \Hom_{\cH_L^-\cap \cH_x}(\cH_x, \psi^{e_L}|_{\cH_L^-\cap \cH_x}).$$
The construction of \cite[Prop. 4.16]{abe:heckemods} shows that $\psi^{e_L}|_{\cH_L^-\cap \cH_x}$ is a twist of the trivial character (cf. \eqref{trivsign}), and therefore it is an injective $\cH_L^-\cap \cH_x$-module (this follows from \cite[Prop. 6.19]{ollivierschneider} or \cite[Thm. 5.2]{norton}).  It follows that $\Hom_{\cH_L^-\cap \cH_x}(\cH_x,\psi^{e_L}|_{\cH_L^-\cap \cH_x})$ is injective as an $\cH_x$-module, and the result follows from Remark \ref{projinjdim}.

(2)  Restricting the \v{C}ech resolution of Lemma \ref{cech} to $\cH_x$ and using part (1) gives a projective resolution
$$\cC_\bullet|_{\cH_x} \longrightarrow I(B,\psi,Q)|_{\cH_x}\longrightarrow 0,$$
and therefore $I(B,\psi,Q)|_{\cH_x}$ has finite projective dimension.  We conclude by using Remark \ref{projinjdim}.  
\end{proof}

\section{Principal series modules -- Type $A_n$}

As a first step, we examine the simple subquotients of principal series modules in type $A_n$.  We will generalize this result in the coming sections.

\begin{propn}\label{Anprop}
Assume the root system of $G$ is of type $A_{n}$, and assume $p\nmid|\Omega_\tor|$.  Let $\mm$ be a simple subquotient of a principal series module for $\cH$.  Then $\pd_{\cH}(\mm)\leq n + r_Z$.  Moreover, when $G$ is semisimple, \eqref{resn} is a projective resolution of $\mm$, and the bound on projective dimension is sharp.  
\end{propn}

\begin{proof}
Since the root system of $G$ is of type $A_{n}$, every vertex in the Bruhat--Tits building is hyperspecial.  Given the vertex $x\in \overline{C}$, we may write $\mm\cong I(B,\psi,Q)$ for some character $\psi$ of $T$ and standard parabolic subgroup $Q$.  Note that $B, \psi,$ and $Q$ implicitly depend on the vertex $x$, which defines a set of positive roots $\Phi^+$.  Choosing another hyperspecial vertex $x'\in \overline{C}$, we obtain an isomorphism $\mm\cong I(B',\psi',Q')$, where $B'$ is the Borel subgroup containing $T$ and defined by $x'$, possibly different from $B$ (this isomorphism follows from considering central characters).  Therefore, Lemma \ref{psproj} shows that $\mm|_{\cH_{y}}$ is projective for every vertex $y\in \overline{C}$, and Lemma \ref{restr} shows that $\mm|_{\cH_{\cF}}$ is projective for every facet $\cF\subset \overline{C}$.  Lemma \ref{redtofacets} and Remark \ref{projinjdim} then give the bound on the projective dimension.

When $G$ is semisimple, we have $r_Z = 0$, so that $\mm|_{\cH_{\cF}^\dagger}(\epsilon_{\cF})$ is projective over $\cH_{\cF}^\dagger$ for every facet $\cF\subset \overline{C}$.  Since induction preserves projectivity, \eqref{resn} is a projective resolution.  Finally, \cite[Cor. 6.12]{ollivierschneider} gives the sharpness of the bound.  
\end{proof}

\begin{ex}
We give an example to show that the condition $p\nmid|\Omega_\tor|$ in Proposition \ref{Anprop} is necessary.  Assume that $F$ is a finite extension of $\bbQ_2$, and let $G = \textnormal{PGL}_2(F)$, so that $\Omega_\tor = \Omega \cong \bbZ/2\bbZ$.  Let $\mm$ denote either the trivial or sign character of $\cH$ (see \eqref{trivsign} below).  The action of $\cH_C^\dagger\cong \overline{\mathbb{F}}_2[\tOmega]$ on $\mm$ factors through a block isomorphic to the algebra $\overline{\mathbb{F}}_2[\Omega]\cong \overline{\mathbb{F}}_2[X]/(X^2)$, and therefore the module $\mm|_{\cH_C^\dagger}$ has infinite projective dimension (see Example \ref{degA1v2} below).  By Lemma \ref{redtofacets}, $\mm$ has infinite projective dimension over $\cH$.  
\end{ex}

\section{Supersingular modules}\label{ssingmods}

We again suppose $G$ is an arbitrary split connected reductive group, and turn our attention to the supersingular modules.  We fix a set $\{\widehat{s_\alpha}\}_{s_\alpha\in S}\subset\tW_\aff$ of lifts of elements of $S$ as in \cite[\S 4.8]{ollivierschneider}.  The affine algebra $\cH_\aff$ is then generated by the elements $\{\T_{t}\}_{t\in T(k_F)}$ and $\{\T_{\widehat{s_\alpha}}\}_{s_\alpha\in S}$, with quadratic relations given by
$$\T_{\widehat{s_\alpha}}^2 = \T_{\widehat{s_\alpha}}\left(\sum_{x\in k_F^\times}\T_{\alpha^\vee(x)}\right).$$
(If $\alpha$ is an affine root of the form $\alpha(\lambda) = \beta(\lambda) + k$ with $\beta\in \Phi, k\in \bbZ$, and $\lambda\in T$, we define $\alpha^\vee:= \beta^\vee$.)  Moreover, each irreducible component $\Pi'$ of $\Pi$ gives rise to a subset $S'$ of $S$ and a subalgebra of $\cH_\aff$, generated by $\{\T_{t}\}_{t\in T(k_F)}$ and $\{\T_{\widehat{s_\alpha}}\}_{s_\alpha\in S'}$; these are called the \emph{irreducible components of $\cH_{\aff}$} (see the discussion preceding \cite[Lem. 6.8]{vigneras:hecke3}).

Recall from \cite[Prop. 2.2]{vigneras:hecke3} that the characters of $\cH_\aff$ are parametrized by pairs $(\xi, J)$, where $\xi:T(k_F)\longrightarrow \fpb^\times$ is a character, and $J\subset S_\xi$, where
$$S_\xi:=\left\{s_\alpha\in S: \xi\circ\alpha^\vee(x) = 1~\textnormal{for all}~x\in k_F^\times\right\}.$$
If $\chi$ is parametrized by $(\xi,J)$, then we have
\begin{eqnarray*}
\chi(\T_t) & = & \xi(t),\\
\chi(\T_{\widehat{s_\alpha}}) & = & \begin{cases}\phantom{-}0 & \textnormal{if}~s_\alpha\not\in J,\\ -1 & \textnormal{if}~s_\alpha\in J.\end{cases}
\end{eqnarray*}

In particular, the \emph{trivial and sign characters of $\cH_{\aff}$} are parametrized by $(\bld,\emptyset)$ and $(\bld,S)$, respectively, where $\bld$ denotes the trivial character of $T(k_F)$.  These characters extend to characters of $\cH$: they are given by
\begin{equation}\label{trivsign}
\T_{\tw}\longmapsto q^{\ell(\tw)} \qquad \textnormal{and}\qquad \T_{\tw}\longmapsto (-1)^{\ell(\tw)},
\end{equation}
respectively (using the convention that $0^0 = 1$).  Given any subalgebra $\cA$ of $\cH$ (e.g., $\cH_{\cF}$, $\cH_{\cF}^\dagger$, etc.), we define the \emph{trivial and sign characters of $\cA$} to be the restrictions of the above characters to $\cA$.   Finally, if $\chi$ is a character of $\cH_\aff$ parametrized by $(\xi,J)$ and $\xi':T(k_F)\longrightarrow \fpb^\times$ is a character such that $S_{\xi'} = S$, we define the \emph{twist of $\chi$ by $\xi'$} to be the character of $\cH_\aff$ parametrized by $(\xi\xi',J)$.

There is a notion of supersingularity for characters of $\cH_\aff$.  We will not give the actual definition here, but merely point out that \cite[Thm. 6.15]{vigneras:hecke3} implies that a character $\chi$ of $\cH_\aff$ is supersingular if and only if its restriction to each irreducible component of $\cH_\aff$ is different from a twist of the trivial or sign character.  When $G = T$ (so that $r_{\textnormal{ss}} = 0$), every character of $\cH_\aff$ is supersingular.

Fix a supersingular character $\chi$.  Let $\tOmega_\chi$ denote the subgroup of $\tOmega$ such that, for every $\widetilde{\omega}\in \tOmega_\chi$, the element $\T_{\widetilde{\omega}}$ fixes $\chi$ under conjugation (note that $\T_{\widetilde{\omega}}$ is invertible and preserves $\cH_\aff$ under conjugation).  We have $T(k_F)\subset \tOmega_\chi$, and $\tOmega_\chi$ is of finite index in $\tOmega$.  We let $\cH_{\aff,\chi}$ denote the subalgebra of $\cH$ generated by $\cH_{\aff}$ and $\{\T_{\widetilde{\omega}}\}_{\widetilde{\omega}\in \tOmega_\chi}$; the braid relations imply that $\cH_{\aff,\chi}$ is free over $\cH_{\aff}$, and $\cH$ is free over $\cH_{\aff,\chi}$ (as either left or right modules).

Now let $\tau$ denote an irreducible finite-dimensional representation of $\tOmega_\chi$ (with $\tOmega_\chi$ acting on the \emph{right}), such that the action of $t\in T(k_F)$ on $\tau$ is equal to the scalar $\chi(\T_{t})$.  We endow $\chi\otimes_{\fpb}\tau$ with a right action of $\cH_{\aff,\chi}$: the element $\T_{\widetilde{w}}\in \cH_{\aff}$ acts by the scalar $\chi(\T_{\widetilde{w}})$, while the elements $\T_{\widetilde{\omega}}$ for $\widetilde{\omega}\in \tOmega_\chi$ act via the representation $\tau$.  Then every simple supersingular right $\cH$-module is isomorphic to
$$\left(\chi\otimes_{\fpb}\tau\right)\otimes_{\cH_{\aff,\chi}}\cH,$$
for a supersingular character $\chi$ of $\cH_{\aff}$ and a finite-dimensional irreducible representation $\tau$ of $\tOmega_\chi$, which agrees with $\chi$ on $T(k_F)$ (see \cite[Thm. 6.18]{vigneras:hecke3}).

\begin{lemma}\label{redtoHaff}
Let $\chi$ be a supersingular character of $\cH_\aff$, and let $\tau$ be an irreducible finite-dimensional representation of $\tOmega_\chi$ whose restriction to $T(k_F)$ agrees with $\chi$ (in the sense above).  Set $\mm := (\chi\otimes_{\fpb}\tau)\otimes_{\cH_{\aff,\chi}}\cH$.  If $\pd_{\cH}(\mm) < \infty$, then $\pd_{\cH_{\aff}}(\chi) < \infty$.  
\end{lemma}

\begin{proof}  By the proof of \cite[Prop. 6.17]{vigneras:hecke3}, we have 
$$\mm|_{\cH_\aff}\cong \left(\bigoplus_{\widetilde{\omega}\in \tOmega_\chi\backslash\tOmega}\chi^{\widetilde{\omega}}\right)^{\oplus\dim_{\fpb}(\tau)},$$ 
where $\chi^{\widetilde{\omega}}$ denotes the character of $\cH_{\aff}$ given by first conjugating an element by $\T_{\widetilde{\omega}}$ and then applying $\chi$.  Since $\cH$ is free over $\cH_{\aff}$, the Eckmann--Shapiro lemma
$$\Ext_{\cH}^i\left(\mm,\Hom_{\cH}(\cH_{\aff}, \nn)\right)\cong \bigoplus_{\widetilde{\omega}\in \tOmega_\chi\backslash\tOmega}\Ext_{\cH_{\aff}}^i(\chi^{\widetilde{\omega}},\nn)^{\oplus\dim_{\fpb}(\tau)}$$
gives the claim.
\end{proof}

Our next goal will be to determine which supersingular characters $\chi$ of $\cH_\aff$ have finite projective dimension.  We require a bit of notation.  Given a character $\xi:T(k_F)\longrightarrow \fpb^\times$, we let $e_\xi\in \cH$ denote the associated idempotent:
\begin{equation}\label{idem}
e_\xi := |T(k_F)|^{-1}\sum_{t\in T(k_F)}\xi(t)\T_{t^{-1}}.
\end{equation}
For $w\in W$, we let $w.\xi$ denote the left action of $W$ on $\xi$ by conjugating the argument (note that the action factors through the projection $W\longtwoheadrightarrow W_0$), and let $e_{w.\xi}$ denote the corresponding idempotent.

We first consider several examples which will arise below.

\begin{ex}\label{degA1}
Let $\cA$ denote the unital four-dimensional algebra over $\fpb$ generated by two primitive orthogonal idempotents $e_1\neq 0,1$ and $e_2 := 1 - e_1$, and an element $\T$ which satisfies
$$e_1\T = \T e_2,~ e_2\T = \T e_1,~\T^2 = 0.$$
The algebra $\cA$ decomposes into principal indecomposable (right) modules as 
$$\cA = e_1\cA \oplus e_2\cA,$$
and if we let $\sigma_j$ denote the (one-dimensional) socle of $e_j\cA$, then $e_j\cA$ is a nonsplit extension of $\sigma_{3 - j}$ by $\sigma_j$.  Splicing these two short exact sequences together gives an infinite projective resolution
$$\ldots \longrightarrow e_{3 - j}\cA\longrightarrow e_{j}\cA\longrightarrow e_{3 - j}\cA\longrightarrow\sigma_j\longrightarrow 0,$$
which may be used to compute
\begin{eqnarray*}
\dim_{\fpb}\left(\Ext^i_{\cA}(\sigma_j, \sigma_j)\right) & = & \begin{cases}1 & \textnormal{if $i$ is even,}\\ 0 & \textnormal{if $i$ is odd,}\end{cases}\\
\dim_{\fpb}\left(\Ext^i_{\cA}(\sigma_j, \sigma_{3 - j})\right) & = & \begin{cases}0 & \textnormal{if $i$ is even,}\\ 1 & \textnormal{if $i$ is odd.}\end{cases}
\end{eqnarray*}
Therefore both $\sigma_1$ and $\sigma_2$ are of infinite projective dimension.  
\end{ex}

\begin{ex}\label{degA1v2}
Let $\cB$ denote the two-dimensional algebra $\fpb[X]/(X^2)$, and let $\sigma$ denote the (one-dimensional) socle of $\cB$.  Then $\sigma$ is the unique simple $\cB$-module, and $\cB$ is a nonsplit extension of $\sigma$ by $\sigma$.  Splicing this short exact sequence with itself gives an infinite projective resolution
$$\ldots \longrightarrow \cB\longrightarrow \cB \longrightarrow \cB\longrightarrow \sigma\longrightarrow 0,$$
which may be used to compute
$$\dim_{\fpb}\left(\Ext_{\cB}^i(\sigma,\sigma)\right) = 1~\textnormal{for all $i\geq 0.$}$$
Therefore $\sigma$ is of infinite projective dimension.  
\end{ex}

\begin{ex}\label{0hecke}
Let $\cC$ denote a $0$-Hecke algebra over $\fpb$ corresponding to an irreducible root system of rank 2 (that is, a $0$-Hecke algebra of type $A_2$, $B_2$ or $G_2$).  The algebra $\cC$ is generated by two elements $\T$ and $\T'$, subject to the relations
$$\T^2 = -\T,~(\T')^2 = -\T',~\underbrace{\T\T'\T\cdots}_{\textnormal{$m$ times}} = \underbrace{\T'\T\T'\cdots}_{\textnormal{$m$ times}},$$
where $m = 3,4,$ or $6$.  By \cite{norton}, the algebra $\cC$ has 4 simple modules, each one-dimensional, given by the characters
$$\T\longmapsto \varepsilon,\qquad \T'\longmapsto\varepsilon',$$
with $\varepsilon,\varepsilon'\in \{0,-1\}$.  Moreover, Theorems 5.1 and 5.2 of \emph{loc. cit.} imply that the only simple modules which are projective are the trivial and sign characters (i.e., corresponding to $(\varepsilon,\varepsilon') = (0,0)$ and $(-1,-1)$, respectively).  Since $\cC$ is a Frobenius algebra, the remaining simple modules must have infinite projective dimension (cf. Remark \ref{projinjdim}).  
\end{ex}

\begin{lemma}\label{redtodegA1}
Let $\chi$ be a supersingular character of $\cH_\aff$, parametrized by the pair $(\xi,J)$.  If $\pd_{\cH_\aff}(\chi) < \infty$, then $S_\xi = S$.
\end{lemma}

\begin{proof}
We may assume $S\neq \emptyset$, and suppose that $S_\xi\neq S$.  Choose $s\in S\smallsetminus S_\xi$, and let $\cF$ denote the codimension 1 facet of $C$ which is fixed by $s$.  Consider first the case $s.\xi\neq \xi$.  The action of $\cH_{\cF}$ on $\chi$ factors through the block $(e_\xi + e_{s.\xi})\cH_{\cF}$, and the latter algebra is of the form considered in Example \ref{degA1} (take $e_1 = e_\xi, e_2 = e_{s.\xi}, \T = (e_\xi + e_{s.\xi})\T_{\widehat{s}}$).  Since $(e_\xi + e_{s.\xi})\cH_{\cF}$ is a block of $\cH_{\cF}$, we obtain
$$\Ext_{\cH_{\cF}}^i\left(\chi|_{\cH_{\cF}}, \chi|_{\cH_{\cF}}\right) = \Ext_{(e_{\xi} + e_{s.\xi})\cH_{\cF}}^i\left(\chi|_{(e_{\xi} + e_{s.\xi})\cH_{\cF}}, \chi|_{(e_{\xi} + e_{s.\xi})\cH_{\cF}}\right)\neq 0$$
for infinitely many $i$.  Therefore $\pd_{\cH_{\cF}}(\chi|_{\cH_{\cF}}) = \infty$, and Lemma \ref{redtofacets} gives $\pd_{\cH_{\aff}}(\chi) = \infty$.

Assume now that $s.\xi = \xi$.  Then the action of $\cH_{\cF}$ factors through the block $e_\xi\cH_{\cF}$, and the latter algebra is of the form considered in Example \ref{degA1v2} (take $e_\xi$ as the unit and $X = e_\xi\T_{\widehat{s}}$).  Once again, $e_\xi\cH_{\cF}$ is a block of $\cH_{\cF}$, and
$$\Ext_{\cH_{\cF}}^i\left(\chi|_{\cH_{\cF}}, \chi|_{\cH_{\cF}}\right) = \Ext_{e_{\xi}\cH_{\cF}}^i\left(\chi|_{e_{\xi}\cH_{\cF}}, \chi|_{e_{\xi}\cH_{\cF}}\right)\neq 0$$
for infinitely many $i$.  As above, we conclude that $\pd_{\cH_{\aff}}(\chi) = \infty$.  
\end{proof}

\begin{lemma}\label{redto0hecke}
Let $\chi$ be a supersingular character of $\cH_{\aff}$, parametrized by the pair $(\xi,J)$, and assume $S_\xi = S$.  If $\pd_{\cH_{\aff}}(\chi) < \infty$, then the root system of $G$ is of type $A_1\times \cdots \times A_1$ (possibly empty product).  
\end{lemma}

\begin{proof}
Assume that the root system of $G$ is not of type $A_1\times\cdots \times A_1$ (so that, in particular, the semisimple rank of $G$ is at least 2).  Then there exists an irreducible component of the affine Dynkin diagram of $G$ which is not of type $\widetilde{A}_1$.  The character $\chi$ corresponds to a labeling of the vertices of this component, with the vertex corresponding to the simple affine root $\alpha$ being labeled by $\chi(\T_{\widehat{s_\alpha}})\in \{0,-1\}$.  Since $\chi$ is supersingular and $S_\xi = S$, there must exist two adjacent vertices with distinct labels.  Let $\cF$ denote the codimension 2 facet of $C$ which is fixed by the simple reflections corresponding to these two vertices.  By assumption, the action of $\cH_{\cF}$ on $\chi$ factors through the algebra $e_\xi\cH_{\cF}$, and the latter algebra is of the form considered in Example \ref{0hecke} (this requires the assumption that the chosen irreducible component is not of type $\widetilde{A}_1$).  By construction of the facet $\cF$, $\chi|_{e_\xi\cH_{\cF}}$ is neither the trivial nor the sign character, and since $e_\xi\cH_{\cF}$ is a product of blocks of $\cH_{\cF}$, we obtain $\pd_{\cH_{\cF}}(\chi|_{\cH_{\cF}}) = \infty$.  Lemma \ref{redtofacets} then gives $\pd_{\cH_{\aff}}(\chi) = \infty$.  
\end{proof}

We may now determine when a supersingular module has finite projective dimension.

\begin{thm}\label{ssing}
Let $\mm = (\chi\otimes_{\fpb}\tau)\otimes_{\cH_{\aff,\chi}}\cH$ be a simple supersingular $\cH$-module, and assume $\chi$ is parametrized by $(\xi,J)$.  Then $\pd_{\cH}(\mm) < \infty$ if and only if the root system of $G$ is of type $A_1\times \cdots\times A_1$ (possibly empty product) and $S_\xi = S$. 
\end{thm}

Note that when the product $A_1\times \cdots \times A_1$ is empty, we have $G = T$ and the condition $S_\xi = S$ is vacuous.

\begin{proof}
Lemmas \ref{redtoHaff}, \ref{redtodegA1}, and \ref{redto0hecke} give the desired conditions on $G$ and $S_\xi$.  We prove the other implication.

Assume that the root system of $G$ is of type $A_1\times\cdots\times A_1$ and $S_\xi = S$.  Given any facet $\cF\in \mathscr{F}_i^\aff$, the second assumption implies that the action of $\cH_{\cF}$ on $\chi$ factors through $e_\xi\cH_{\cF}$, which is a $0$-Hecke algebra of type $A_1\times\cdots\times A_1$ and is therefore semisimple.  By Lemma \ref{redtofacets}, we conclude that $\pd_{\cH_{\aff}}(\chi) < \infty$.

We claim that $\tOmega_\chi$ is the subgroup generated by $T(k_F)$ and (the image of) $Z/(Z\cap I(1))$.  Note first that the set $S$ admits a partition into commuting subsets
$$S = \bigsqcup_{\alpha\in \Pi}\left\{s_{\alpha},~s_{\alpha}\alpha^\vee(\varpi) =: s_\alpha'\right\},$$
where $\varpi$ is a fixed choice of uniformizer of $F$.  The group $\Omega$ acts on $S$ by conjugation, and writing an element of $\Omega$ with respect to the decomposition $W = W_0\ltimes \Lambda$ shows that $\Omega$ acts on each set $\{s_{\alpha},~s_{\alpha}'\}$.  Next, by definition of supersingularity, the values $\chi(\T_{\widehat{s_\alpha}}), \chi(\T_{\widehat{s_\alpha'}})\in\{0,-1\}$ are distinct for every choice of $\alpha\in\Pi$.  Therefore, if $\widetilde{\omega}\in \tOmega_\chi$, we have
$$\chi(\T_{\widehat{s_\alpha}}) = \chi^{\widetilde{\omega}}(\T_{\widehat{s_\alpha}}) = \chi(\T_{\widetilde{\omega}\widehat{s_\alpha}\widetilde{\omega}^{-1}}) \quad\textnormal{and}\quad \chi(\T_{\widehat{s_\alpha'}}) = \chi^{\widetilde{\omega}}(\T_{\widehat{s_\alpha'}}) = \chi(\T_{\widetilde{\omega}\widehat{s_\alpha'}\widetilde{\omega}^{-1}})$$
for all $\alpha\in \Pi$, which implies $\widetilde{\omega}\widehat{s_\alpha}\widetilde{\omega}^{-1} = \widehat{s_\alpha}t_\alpha$ and $\widetilde{\omega}\widehat{s_\alpha'}\widetilde{\omega}^{-1} = \widehat{s_\alpha'}t_\alpha'$ for some $t_\alpha, t_\alpha'\in T(k_F)$ by the two comments above.  Letting $\omega\in \Omega$ denote the image of $\widetilde{\omega}$ under projection, the previous sentence implies that $\omega$ commutes with all of $W_\aff$, and therefore must lie in $Z/(Z\cap I)$.  This gives the claim.

By choosing a splitting of
$$1\longrightarrow T(k_F)\longrightarrow \tOmega_\chi\longrightarrow Z/(Z\cap I)\cong \bbZ^{\oplus r_Z}\longrightarrow 1,$$
we see that $\cH_{\aff,\chi}$ is a Laurent polynomial algebra over $\cH_{\aff}$, with basis given by the elements $\{\T_{\widehat{z}}\}$, where $\widehat{z}$ is in the image of the splitting.  \cite[Prop. 7.5.2]{mcconnellrobson} now implies
\begin{eqnarray*}
\pd_{\cH_{\aff,\chi}}\left(\chi\otimes_{\fpb}\tau\right) & \leq & \pd_{\cH_{\aff}}\left((\chi\otimes_{\fpb}\tau)|_{\cH_{\aff}}\right) + r_Z\\
 & = & \pd_{\cH_{\aff}}\left(\chi^{\oplus\dim_{\fpb}(\tau)}\right) + r_Z < \infty.
\end{eqnarray*}
Finally, by the Eckmann--Shapiro lemma, we get
$$\Ext_{\cH}^i\left(\mm,\nn\right)\cong \Ext_{\cH_{\aff,\chi}}^i\left(\chi\otimes_{\fpb}\tau,\nn|_{\cH_{\aff,\chi}}\right),$$
which shows that $\pd_{\cH}(\mm) < \infty$.  
\end{proof}

\section{Parabolic coinduction}\label{sectionparabind}

We first show how to transfer information about projective dimension using parabolic coinduction.

\begin{lemma}\label{parabind}
Let $P = M\ltimes N$ denote a standard parabolic subgroup of $G$, and let $\nn$ be a right $\cH_M$-module.  Then we have $\pd_{\cH_M}(\nn) < \infty$ if and only if $\pd_{\cH}(I_M(\nn)) < \infty$.
\end{lemma}

\begin{proof}
Remark \ref{projinjdim} implies that it suffices to prove the claim for injective dimensions.  By \cite[Prop. 4.1]{vigneras:hecke5} the parabolic coinduction functor $I_{M}(-)$ from the category of right $\cH_M$-modules to the category of right $\cH$-modules admits a left adjoint, which we denote by $L_M(-)$ (N.B.: in \emph{loc. cit.}, the functors $I_M(-)$ and $L_M(-)$ are denoted $\bbI_{\cH_M}^{\cH}(-)$ and $\bbL_{\cH_M}^{\cH}(-)$, respectively).  Moreover, Sections 3 and 4 of \emph{loc. cit.} show that $L_M(-)$ is given by localization at a central non-zero divisor of $\cH_M^-$, and is therefore exact.  Hence, we obtain
$$\Ext^i_{\cH}\left(\mm, I_{M}(\nn)\right)\cong \Ext^i_{\cH_M}\left(L_M(\mm), \nn\right),$$
which shows that $\id_{\cH_M}(\nn) < \infty$ implies $\id_{\cH}(I_M(\nn)) < \infty$.

Now let $\vv$ be an arbitrary $\cH_M$-module.  Using the explicit description of the functors $I_M(-)$ and $L_M(-)$, we see that $L_M(I_M(\vv))\cong \vv$ (see also \cite[Prop. 4.12]{abe:heckemods}).  This gives
$$\Ext^i_{\cH}\left(I_M(\vv), I_{M}(\nn)\right)\cong \Ext^i_{\cH_M}\left(\vv, \nn\right),$$
which shows that $\id_{\cH}(I_M(\nn))< \infty$ implies $\id_{\cH_M}(\nn) < \infty$.  
\end{proof}

\begin{remark}
Let $\psi:T\longrightarrow \fpb^\times$ be a smooth character of $T$, and use the same letter to denote the associated right $\cH_T$-module.  Then $\psi$ is naturally a module over the block $e_{\psi|_{T(k_F)}}\cH_T\cong \fpb[X_1^{\pm 1},\ldots, X_{r_{\textnormal{ss}} + r_Z}^{\pm 1}]$.  This algebra has global dimension $r_{\textnormal{ss}} + r_Z$ (\cite[Thm. 7.5.3(iv)]{mcconnellrobson}), and the above lemma implies that $\pd_{\cH}(I_T(\psi)) < \infty$.  By using a Koszul resolution of the $\cH_T$-module $\psi$, we easily obtain
$$\dim_{\fpb}\left(\Ext_{\cH_T}^{i}(\psi,\psi)\right) = \binom{r_{\textnormal{ss}} + r_Z}{i},$$
and therefore the proof above shows
$$\Ext_{\cH}^{r_{\textnormal{ss}} + r_Z}\left(I_T(\psi),I_T(\psi)\right)\neq 0.$$
Hence $r_{\textnormal{ss}} + r_Z \leq \pd_{\cH}(I_T(\psi)) < \infty$.  Using Remark \ref{projinjdim}, this shows that the bound on the self-injective dimension of $\cH$ obtained in \cite{ollivierschneider} is sharp, i.e., $\id_{\cH}(\cH) = r_{\textnormal{ss}} + r_Z$.  
\end{remark}

Before we proceed, we require a simple lemma.

\begin{lemma}\label{pi1s}
Let $M$ be a standard Levi subgroup of $G$, and let $\Omega_M \cong \Omega_{M,\tor}\times\Omega_{M,\free}$ denote the length 0 subgroup of $W_M$ (relative to the length function on $W_M$).  If $p\nmid|\Omega_{\tor}|$, then $p\nmid|\Omega_{M,\tor}|$.  
\end{lemma}

\begin{proof}
By \cite[\S 1.1 and Prop. 1.10]{borovoi} (see also \cite[Prop. 3.36]{vigneras:hecke1}), the group $\Omega$ is isomorphic to the quotient of $X_*(T)$ by the subgroup generated by the coroots of $G$ (and likewise for $\Omega_M$).  Therefore, we have a surjection $\Omega_M\longtwoheadrightarrow \Omega$, which is easily seen to be injective when restricted to $\Omega_{M,\tor}$.  The result follows.  
\end{proof}

\begin{propn}\label{extn}
Assume $p\nmid|\Omega_{\tor}|$.  Let $P = M\ltimes N$ be a standard parabolic subgroup of $G$, and let $\nn$ be a simple supersingular right $\cH_M$-module.  Let $Q = L\ltimes V$ denote another standard parabolic subgroup such that $P\subset Q\subset P(\nn)$.  Then $\pd_{\cH_M}(\nn) < \infty$ if and only if $\pd_{\cH_L}(\nn^{e_L}) < \infty$.
\end{propn}

\begin{proof}
By Lemma \ref{pi1s}, we may suppose that $Q = G$.  The set $\Pi$ then admits an orthogonal decomposition $\Pi = \Pi_M \sqcup \Pi_2$, where $\langle\beta,\alpha^\vee\rangle = 0$ for every $\beta\in \Pi_M, \alpha\in \Pi_2$, and such that $\T^M_{\alpha^\vee(x)}$ acts trivially on $\nn$ for every $x\in F^\times, \alpha\in \Pi_2$.  Moreover, we obtain a partition $S = S_M \sqcup S_2$ on the set of affine reflections (cf. \cite[\S 4.4]{abe:heckemods}).  Since this partition comes from the orthogonal decomposition of $\Pi$, the elements of $S_M$ commute with the elements of $S_2$.

Assume first that $\pd_{\cH_M}(\nn) = \infty$.  By Lemmas \ref{redtofacets} and \ref{pi1s}, there must exist some facet $\cF_M$ in the semisimple Bruhat--Tits building of $M$ (in the closure of the chamber corresponding to $I\cap M$) such that $\nn|_{\cH_{M,\cF_M}}$ is not projective.  Let $S_{\cF_M}\subset S_M$ denote the set of simple reflections which fix $\cF_M$ pointwise, so that $\cH_{M,\cF_M}$ is generated by $\{\T^M_{\widehat{s}}\}_{s\in S_{\cF_M}}$ and $\{\T^M_t\}_{t\in T(k_F)}$.  Now view $S_{\cF_M}$ as a subset of $S$, and let $\cF$ denote the facet of $C$ fixed pointwise by $S_{\cF_M}$.  The algebra $\cH_{\cF}$ is generated by $\{\T_{\widehat{s}}\}_{s\in S_{\cF_M}}$ and $\{\T_t\}_{t\in T(k_F)}$.  One easily checks that the elements $\widehat{s}$ for $s\in S_{\cF_M}$ are all $M$-negative (see \cite[\S 4.1]{abe:heckemods}; this uses the fact that we have an \emph{orthogonal} decomposition).  
Therefore, we have $\cH_{M,\cF_M}\subset \cH_M^-$, and the map $j_M^-:\cH_M^-\longhookrightarrow \cH$ induces an algebra isomorphism $\cH_{M,\cF_M}\stackrel{\sim}{\longrightarrow} \cH_{\cF}$.  By \cite[Prop. 4.16]{abe:heckemods}, the algebra $\cH_{\cF}$ acts on $\nn^{e_G}$ through the isomorphism $j_M^-$, and we conclude that $\nn^{e_G}|_{\cH_{\cF}}$ is not projective.  Hence $\pd_{\cH}(\nn^{e_G}) = \infty$.

Assume now that $\pd_{\cH_{M}}(\nn) < \infty$, and write $\nn = (\chi\otimes_{\fpb}\tau)\otimes_{\cH_{M,\aff,\chi}}\cH_M$ as in Section \ref{ssingmods}.  By Theorem \ref{ssing}, we have $S_{M,\xi} = S_M$, where $(\xi,J)$ parametrizes the character $\chi$ of $\cH_{M,\aff}$.  Since $\nn$ extends to $\cH$ by assumption, we have $S_\xi = S$ (cf. \cite[Prop. 4.16]{abe:heckemods}).

Fix a facet $\cF\subset \overline{C}$, and let $S_{\cF}\subset S$ denote the set of simple reflections which fix $\cF$ pointwise.  By definition of the extension $\nn^{e_G}$, the action of $\cH_{\cF}$ on $\nn^{e_G}$ factors through $e_\xi\cH_{\cF}$.  The latter algebra decomposes as a tensor product
$$e_\xi\cH_{\cF}\cong \cH_1\otimes_{\fpb}\cH_2,$$
where $\cH_1$ is the algebra generated by $\{e_\xi\T_{\widehat{s}}\}_{s\in S_M\cap S_{\cF}}$ and $\cH_2$ is the algebra generated by $\{e_\xi\T_{\widehat{s}}\}_{s\in S_2\cap S_{\cF}}$.  As above, $\cH_1$ is isomorphic via $j_M^-$ to a finite Hecke algebra of the form $e_\xi\cH_{M,\cF_M}$, where $\cF_M$ is the facet in the semisimple Bruhat--Tits building of $M$ fixed by $S_M\cap S_\cF$.  Again using \cite[Prop. 4.16]{abe:heckemods}, the $e_\xi\cH_{\cF}$-module $\nn^{e_G}|_{e_\xi\cH_{\cF}}$ decomposes as a(n external) tensor product
$$\nn^{e_G}|_{e_\xi\cH_{\cF}} \cong \nn|_{e_\xi\cH_{M,\cF_M}}\otimes_{\fpb}\chi_{\textnormal{triv},2},$$
where $\chi_{\textnormal{triv},2}$ denotes the trivial character of $\cH_2$.  Since $\pd_{\cH_M}(\nn) < \infty$, the restriction of $\nn$ to $\cH_{M,\cF_M}$ must be projective (this again uses Lemmas \ref{redtofacets} and \ref{pi1s}).  Since $\chi_{\textnormal{triv},2}$ is projective over $\cH_2$ (cf. \cite[Thm. 5.2]{norton}), we see that $\nn^{e_G}|_{\cH_{\cF}}$ is projective.  Using Lemma \ref{redtofacets} one final time gives $\pd_{\cH}(\nn^{e_G}) < \infty$.  
\end{proof}

We now arrive at our main result.

\begin{thm}\label{main}
Assume $p\nmid|\Omega_{\tor}|$.  Let $P = M\ltimes N$ be a standard parabolic subgroup of $G$, and let $\nn$ be a simple supersingular right $\cH_M$-module.  Let $Q = L\ltimes V$ denote another standard parabolic subgroup such that $P\subset Q\subset P(\nn)$.  Then $\pd_{\cH_M}(\nn) < \infty$ if and only if $\pd_{\cH}(I(P,\nn,Q)) < \infty$.  
\end{thm}

\begin{proof}
Since parabolic coinduction is exact and transitive (cf. \cite[Cor. 1.10]{vigneras:hecke5}), we have 
$$I(P,\nn,Q) = I_{M(\nn)}\left(I_{\cH_{M(\nn)}}(P\cap M(\nn),\nn,Q\cap M(\nn))\right),$$
where $I_{\cH_{M(\nn)}}(P\cap M(\nn),\nn,Q\cap M(\nn))$ is a simple $\cH_{M(\nn)}$-module defined in the same manner as $I(P,\nn,Q)$.  Therefore, by Lemmas \ref{parabind} and \ref{pi1s} we may assume $P(\nn) = G$, so that $\Pi$ admits an orthogonal decomposition $\Pi = \Pi_M \sqcup \Pi_2$ and $S$ admits a partition into commuting subsets $S = S_M \sqcup S_2$ (as in Proposition \ref{extn}).

Assume first that $\pd_{\cH_M}(\nn) < \infty$.  By Lemma \ref{parabind} and the proposition above, we have $\pd_{\cH}(I_{L'}(\nn^{e_{L'}})) < \infty$ for every parabolic subgroup $Q' = L'\ltimes V'$ such that $P\subset Q' \subset G$.  Moreover, the \v{C}ech resolution of Lemma \ref{cech} gives rise to a hyper-Ext spectral sequence
$$E_1^{i,j} = \bigoplus_{\sub{L\subset L'\subset G}{|\Pi_{L'}\smallsetminus\Pi_L| = i}}\Ext_{\cH}^j\left(I_{L'}(\nn^{e_{L'}}),\vv\right)\Longrightarrow \Ext_{\cH}^{i + j}\left(I(P,\nn,Q),\vv\right).$$
Since $E_1^{i,j}$ vanishes for $i$ and $j$ sufficiently large (independent of $\vv$), we conclude $\pd_{\cH}(I(P,\nn,Q)) < \infty$.

Fix now a parabolic subgroup $Q' = L'\ltimes V'$ such that $P\subset Q'\subset G$ and consider the coinduced module $I_{L'}(\nn^{e_{L'}}) = \Hom_{\cH_{L'}^-}(\cH,\nn^{e_{L'}})$.  By \cite[Lem. 4.10]{abe:heckemods}, we have an isomorphism of vector spaces
\begin{eqnarray*}
I_{L'}(\nn^{e_{L'}}) & \stackrel{\sim}{\longrightarrow}  &  \bigoplus_{v\in W_0^{L'}}\nn\\
\varphi & \longmapsto & \left(\varphi(\T_{\widehat{v}})\right)_{v\in W_0^{L'}}. 
\end{eqnarray*}
In the above, if $v = s_{\alpha_1}\cdots s_{\alpha_k}$ is a reduced expression for $v\in W_0$ with $s_{\alpha_i}\in S\cap W_0$, we define $\widehat{v} := \widehat{s_{\alpha_1}}\cdots\widehat{s_{\alpha_k}}$; \cite[Props. 8.8.3 and 9.3.2]{springer} imply this element is well-defined.

The orthogonal decomposition of $\Pi$ implies that $W_0$ is a product of two finite Weyl groups $W_0 \cong W_{M,0}\times W_{2,0}$ (corresponding to the sets of generators $S_M\cap W_0$ and $S_2\cap W_0$, respectively).  This easily implies that the set $W_0^{L'}$ is contained in $W_{2,0}$, and therefore commutes with $S_M$.  Using \cite[Prop. 9.3.2]{springer} again, we see that in fact the elements $\widehat{v}$ and $\widehat{s}$ commute, where $v\in W_0^{L'}$ and $s\in S_M$.  This implies that the isomorphism above is equivariant for the operators $\{\T_{\widehat{s}}\}_{s\in S_M}$, where $\T_{\widehat{s}}$ acts on the right-hand side as $\T_{\widehat{s}}^M$ via $j_M^-$.  Moreover, by the fact that $\T_{\alpha^\vee(x)}^M$ acts trivially on $\nn$ for all $\alpha\in \Pi_2$ and $x\in k_F^\times$, we see that the actions of $\T_{t}$ on the left-hand side and $\T_t^M$ on the right-hand side agree, so that the isomorphism is equivariant for the operators $\{\T_t\}_{t\in T(k_F)}$.

Assume now that $\pd_{\cH_M}(\nn) = \infty$.  As in the proof of Proposition \ref{extn}, there exist facets $\cF_M$ in the semisimple Bruhat--Tits building of $M$ and $\cF\subset \overline{C}$ such that $\nn|_{\cH_{M,\cF_M}}$ is not projective and such that $j_M^-$ induces an algebra isomorphism $\cH_{M,\cF_M}\stackrel{\sim}{\longrightarrow}\cH_{\cF}$.  By the discussion of the preceding paragraph, the isomorphism $I_{L'}(\nn^{e_{L'}})\cong \bigoplus_{v\in W_0^{L'}} \nn$ is an isomorphism of $\cH_{\cF}$-modules, where $\cH_{\cF}$ acts on the right-hand side as $\cH_{M,\cF_M}$ via $j_M^-$.  Consequently, we see that $I(P,\nn,Q)|_{\cH_{\cF}}$ is isomorphic to a direct sum of copies of $\nn$ (where we again view $\nn$ as a $\cH_{\cF}$-module via $j_M^-$), and so $I(P,\nn,Q)|_{\cH_{\cF}}$ is not projective.  By Lemma \ref{redtofacets}, we conclude that $\pd_{\cH}(I(P,\nn,Q)) = \infty$.
\end{proof}

\begin{cor}\label{maincor}
Assume $p\nmid|\Omega_{\tor}|$, and let $\mm$ be a simple right $\cH$-module.  Write $\mm\cong I(P,\nn,Q)$, with $\nn \cong (\chi\otimes_{\fpb}\tau)\otimes_{\cH_{M,\aff,\chi}}\cH_M$ and $\chi$ parametrized by $(\xi,J)$.  Then the following are equivalent:
\begin{itemize}
\item $\pd_{\cH}(\mm) < \infty$;
\item $\pd_{\cH_M}(\nn) < \infty$;
\item the root system of $M$ is of type $A_1\times \cdots \times A_1$ (possibly empty product) and $S_{M,\xi} = S_M$.
\end{itemize}
Moreover, when $G$ is semisimple and $\pd_{\cH}(\mm) < \infty$, \eqref{resn} is a projective resolution of $\mm$, and $\pd_{\cH}(\mm) = r_{\textnormal{ss}}$.  
\end{cor}

\begin{proof}
The first assertion follows from combining Theorems \ref{main} and \ref{ssing}.  Assuming that $G$ is semisimple gives $r_Z = 0$, and Lemma \ref{redtofacets} shows that $\pd_{\cH}(\mm)<\infty$ is equivalent to $\mm|_{\cH_{\cF}^\dagger}(\epsilon_{\cF})$ being projective for every facet $\cF\subset \overline{C}$.  Since induction preserves projectivity, we get that \eqref{resn} is a projective resolution, and \cite[Cor. 6.12]{ollivierschneider} gives the exact value of $\pd_{\cH}(\mm)$.  
\end{proof}

\begin{remark}
Since simple $\cH$-modules are finite-dimensional, they possess a central character.  Using this fact and a slightly stronger restriction on $p$ than in the corollary above, we can actually prove that $\pd_{\cH}(\mm) = r_{\textnormal{ss}} + r_{Z}$ whenever $\pd_{\cH}(\mm) < \infty$; see Proposition \ref{zcharprop} below.  
\end{remark}

\begin{remark}
Theorem \ref{main} gives a slightly different proof of Proposition \ref{Anprop}.  We have chosen to keep the proof of Proposition \ref{Anprop} intact, in the hopes that the techniques used therein (especially Lemma \ref{mackey}) may find application elsewhere.  
\end{remark}

\section{Complements}\label{comp}

\subsection{Iwahori--Hecke modules}\label{iwahori} 
  As a special case of the above results, we now discuss projective dimensions of Iwahori--Hecke modules.

Let $\cH'$ denote the \emph{Iwahori--Hecke algebra} over $\fpb$, defined with respect to the subgroup $I$ of $G$:
$$\cH' := \End_G\left(\cind_I^G(\bld)\right),$$
where $\bld$ now denotes the trivial $I$-module over $\fpb$.  Given an algebra related to $\cH$, we denote with a prime the analogously defined algebra for $\cH'$ (so that $\cH_M'$ denotes the Iwahori--Hecke algebra of a Levi subgroup $M$ with respect to $I\cap M$, $\cH_\cF'$ is the subalgebra of $\cH'$ defined by $\End_{\cP_\cF}(\cind_I^{\cP_\cF}(\bld))$, etc.).  Recall the (primitive) central idempotent $e_{\bld} \in \cH$ defined in equation \eqref{idem}:
$$e_{\bld} = |T(k_F)|^{-1}\sum_{t\in T(k_F)}\T_t.$$
Using $e_{\bld}$ we may identify $\cH'$ with the subalgebra $e_{\bld}\cH$ (which is a block of $\cH$).  Likewise, we make the identifications $\cH_M' = e_{\bld}\cH_M, \cH_{\cF}' = e_{\bld}\cH_{\cF}$, etc..  See \cite{vigneras:hecke1} for more details.

Now let $\mm$ denote a right $\cH$-module.  If $\mm\cdot e_{\bld} = \mm$, then we may naturally view $\mm$ as an $\cH'$-module, and conversely every $\cH'$-module arises in this way.  In particular, let $P = M\ltimes N$ be a standard parabolic subgroup, and $\nn := (\chi\otimes_{\fpb}\tau)\otimes_{\cH_{M,\aff,\chi}}\cH_M$ a simple supersingular right $\cH_M$-module.  One easily sees that the simple right $\cH$-module $\mm := I(P,\nn,Q)$ satisfies $\mm\cdot e_{\bld} = \mm$ if and only if $\nn\cdot e_{\bld} = \nn$, if and only if $\chi$ is parametrized by $(\bld, J)$ for some $J\subset S_M$.  Corollary \ref{maincor} now takes the following form, noting that $\pd_{\cH}(\mm) < \infty$ is equivalent to $\pd_{\cH'}(\mm) < \infty$ for right $\cH$-modules satisfying $\mm\cdot e_{\bld} = \mm$.

\begin{cor}\label{iwahoripd}
Assume $p\nmid|\Omega_{\tor}|$, and let $\mm$ be a simple right $\cH'$-module.  Write $\mm\cong I(P,\nn,Q)$, with $\nn \cong (\chi\otimes_{\fpb}\tau)\otimes_{\cH_{M,\aff,\chi}}\cH_M$ and $\chi$ parametrized by $(\bld,J)$.  Then the following are equivalent:
\begin{itemize}
\item $\pd_{\cH'}(\mm) < \infty$;
\item $\pd_{\cH_M'}(\nn) < \infty$;
\item the root system of $M$ is of type $A_1\times \cdots \times A_1$ (possibly empty product).
\end{itemize}
Moreover, when $G$ is semisimple and $\pd_{\cH'}(\mm) < \infty$, the ``$\cH'$ version'' of \eqref{resn} is a projective resolution of $\mm$, and $\pd_{\cH'}(\mm) = r_{\textnormal{ss}}$.  
\end{cor}

\subsection{Projective resolutions of $G$-representations}\label{represns}  
In this subsection, we take $p>2$, and let $G$ be equal to either $\textnormal{PGL}_2(\qp)$ or $\textnormal{SL}_2(\qp)$.  In this case (cf. \cite{ollivier:foncteur} and \cite{koziol:glnsln}), the category $\mathfrak{Mod}\textrm{-}\cH$ of right $\cH$-modules is equivalent to the category $\mathfrak{Rep}^{I(1)}_{\fpb}(G)$ of $\fpb$-representations of $G$ generated by their $I(1)$-invariant vectors.  Explicitly, this equivalence is given by the pair of adjoint functors
\begin{eqnarray*}
\mathfrak{Mod}\textrm{-}\cH & \cong & \mathfrak{Rep}^{I(1)}_{\fpb}(G)\\
\mm & \longmapsto & \mm\otimes_{\cH}\cind_{I(1)}^G(\bld)\\
\pi^{I(1)} & \longmapsfrom & \pi.
\end{eqnarray*}

Let $\pi$ be a smooth irreducible representation of $G$ which is either an irreducible subquotient of a principal series representation, or a supersingular representation which satisfies $\pi^I\neq 0$ (see \cite{breuil:gl2qp}, \cite{abdellatif:sl2qp}).  The nonzero vector space $\pi^{I(1)}$ then becomes a simple right $\cH$-module, and Corollary \ref{maincor} implies that it has projective dimension 1 over $\cH$ (for supersingular representations satisfying $\pi^I \neq 0$, we have $\pi^{I(1)} = \pi^I$, and thus we may instead apply Corollary \ref{iwahoripd}).  We note that, if $\pi'$ is an irreducible supersingular representation of $G$ for which the associated $\cH$-module $(\pi')^{I(1)}$ satisfies $S_\xi = \emptyset$, then $(\pi')^{I(1)}$ has infinite projective dimension in $\mathfrak{Mod}\textrm{-}\cH$, and consequently $\pi'$ has infinite projective dimension in $\mathfrak{Rep}^{I(1)}_{\fpb}(G)$.

Let $\pi$ be as above, and let $x$ and $x'$ denote the two vertices in the closure of the chamber $C$.   By Corollary \ref{maincor}, we obtain a projective resolution of $\pi^{I(1)}$ given by 
$$0\longrightarrow \pi^{I(1)}(\epsilon_C)\otimes_{\cH_C^\dagger}\cH\longrightarrow \pi^{I(1)}\otimes_{\cH_x}\cH\longrightarrow \pi^{I(1)}\longrightarrow 0$$
for $G = \textnormal{PGL}_2(\qp)$, and by
$$0\longrightarrow \pi^{I(1)}\otimes_{\cH_C}\cH\longrightarrow \left(\pi^{I(1)}\otimes_{\cH_x}\cH\right)\oplus \left(\pi^{I(1)}\otimes_{\cH_{x'}}\cH\right) \longrightarrow \pi^{I(1)}\longrightarrow 0$$
for $G = \textnormal{SL}_2(\qp)$ (here we identify the semisimple buildings of the groups $\textnormal{PGL}_2(\qp)$ and $\textnormal{SL}_2(\qp)$).  One easily checks that for any facet $\cF\subset \overline{C}$, we have a $G$-equivariant isomorphism 
$$\pi^{I(1)}(\epsilon_{\cF})\otimes_{\cH_{\cF}^\dagger}\cind_{I(1)}^G(\bld)\cong \cind_{\cP_{\cF}^\dagger}^G\left(\langle\cP_{\cF}^\dagger.\pi^{I(1)}\rangle\otimes_{\fpb}\epsilon_{\cF}\right),$$
where $\langle\cP_{\cF}^\dagger.\pi^{I(1)}\rangle$ denotes the $\cP_{\cF}^\dagger$-subrepresentation of $\pi$ generated by $\pi^{I(1)}$.  
Applying the equivalence of categories above, we thus obtain a resolution
\begin{equation}\label{pglres}
0\longrightarrow \cind_{\cP_{C}^\dagger}^G\left(\pi^{I(1)}\otimes_{\fpb}\epsilon_{C}\right)\longrightarrow \cind_{\cP_{x}}^G\left(\langle\cP_x.\pi^{I(1)}\rangle\right)\longrightarrow \pi\longrightarrow 0
\end{equation}
when $G = \textnormal{PGL}_2(\qp)$, and 
\begin{equation}\label{slres}
0\longrightarrow \cind_{I}^G\left(\pi^{I(1)}\right)\longrightarrow \cind_{\cP_x}^G\left(\langle\cP_x.\pi^{I(1)}\rangle\right)~ \oplus~ \cind_{\cP_{x'}}^G\left(\langle\cP_{x'}.\pi^{I(1)}\rangle\right) \longrightarrow \pi \longrightarrow 0 
\end{equation}
when $G = \textnormal{SL}_2(\qp)$ (note that in both cases, $\pi^{I(1)}$ is naturally a representation of $\cP_C^\dagger$).  Collecting everything gives the following result.

\begin{propn}
Let $G = \textnormal{PGL}_2(\qp)$ (resp. $G = \textnormal{SL}_2(\qp)$), with $p> 2$.  Let $\pi$ denote an irreducible subquotient of a principal series representation of $G$, or an irreducible supersingular representation of $G$ which satisfies $\pi^I\neq 0$.  Then \eqref{pglres} (resp. \eqref{slres}) is a projective resolution of $\pi$ in the abelian category $\mathfrak{Rep}^{I(1)}_{\fpb}(G)$.  
\end{propn}

\begin{remark}
Let $\pi$ be equal to the trivial representation of $G$.  The terms in the resolutions \eqref{pglres} and \eqref{slres} take the form $\cind_{\cP_{\cF}^\dagger}^G(\epsilon_{\cF})$, and we have
\begin{eqnarray*}
\Hom_{G}\left(\cind_{\cP_{\cF}^\dagger}^G(\epsilon_{\cF}), \tau\right) & \cong & \Hom_{\cP_{\cF}^\dagger}(\epsilon_{\cF}, \tau|_{\cP_{\cF}^\dagger}) \\
 & \cong & \tau^{\cP_{\cF}^\dagger, \epsilon_{\cF}}\\
 & := & \left\{v\in \tau: g.v = \epsilon_{\cF}(g)v~\textnormal{for all}~g\in \cP_{\cF}^\dagger\right\},
\end{eqnarray*}
where $\tau$ is a smooth $G$-representation.  Since the coefficient field has characteristic $p$, the functor $\tau\longmapsto \tau^{\cP_{\cF}^\dagger, \epsilon_{\cF}}$ will not be exact in general, and therefore the resolutions \eqref{pglres} and \eqref{slres} will not give projective resolutions in the entire category $\mathfrak{Rep}_{\fpb}(G)$ of smooth $G$-representations.  
\end{remark}

\begin{remark}
By direct inspection, one easily checks that if $\pi$ is an irreducible subquotient of a principal series representation of $\textnormal{PGL}_2(\qp)$ or $\textnormal{SL}_2(\qp)$ and $y$ a vertex in the closure of $\overline{C}$, then 
$$\langle\cP_{y}.\pi^{I(1)}\rangle \cong \pi^{\cP_y(1)}$$
as representations of $\cP_y$, where $\cP_y(1)$ denotes the pro-$p$ radical of $\cP_y$.  Therefore, the projective resolutions \eqref{pglres} and \eqref{slres} take the form
$$0\longrightarrow \cind_{\cP_{C}^\dagger}^G\left(\pi^{I(1)}\otimes_{\fpb}\epsilon_{C}\right)\longrightarrow \cind_{\cP_{x}}^G\left(\pi^{\cP_x(1)}\right)\longrightarrow \pi\longrightarrow 0$$
when $G = \textnormal{PGL}_2(\qp)$, and 
$$0\longrightarrow \cind_{I}^G\left(\pi^{I(1)}\right)\longrightarrow \cind_{\cP_x}^G\left(\pi^{\cP_x(1)}\right)~ \oplus~ \cind_{\cP_{x'}}^G\left(\pi^{\cP_{x'}(1)}\right)\longrightarrow \pi \longrightarrow 0$$
when $G = \textnormal{SL}_2(\qp)$.  In other words, the representation $\pi$ is the $0^{\textnormal{th}}$ homology of the coefficient system denoted $\underline{\underline{\pi}}$ in \cite[\S II.2]{ss:reptheorysheaves}.  
\end{remark}

\subsection{Central characters}\label{zcharsection}

One can also ask about the behavior of the resolution \eqref{resn} when the module $\mm$ possesses a central character.  We will show that, by passing to an appropriate subcategory of modules with a fixed ``central character'' (in a sense to be made precise below), \eqref{resn} becomes a projective resolution whenever $\pd_{\cH}(\mm) < \infty$.

Recall that $Z$ denotes the connected center of $G$.  Let us fix a splitting of the short exact sequence
$$1\longrightarrow (Z\cap I)/(Z\cap I(1)) \longrightarrow Z/(Z\cap I(1))\longrightarrow Z/(Z\cap I)\cong \bbZ^{\oplus r_Z}\longrightarrow 1,$$
and for $z\in Z/(Z\cap I)$, we let $\widehat{z}\in Z/(Z\cap I(1)) \longhookrightarrow \tW$ denote its image under the splitting.  Since the elements of $Z/(Z\cap I(1))$ (considered as elements of $\tW$) have length zero, we see that the $\fpb$-vector space spanned by $\{\T_{\widehat{z}}\}_{z\in Z/(Z\cap I)}$ is a subalgebra of the center of $\cH$.  Denote this algebra by $\cY$, and note that $\cY\subset \cH_{\cF}^\dagger$ for every facet $\cF\subset \overline{C}$.

Fix now a character $\zeta:\cY\longrightarrow \fpb$, and define the quotient algebras
$$\cH^\zeta := \cH/\left(\T_{\widehat{z}} - \zeta(\T_{\widehat{z}})\right)_{z\in Z/(Z\cap I)}\cH,$$
$$\cH_{\cF}^{\dagger, \zeta} := \cH_{\cF}^{\dagger}/\left(\T_{\widehat{z}} - \zeta(\T_{\widehat{z}})\right)_{z\in Z/(Z\cap I)}\cH_{\cF}^\dagger,$$
where $\cF$ is a facet in the closure of $C$ (since $\cY$ is central, the ideals are two-sided).  For a fixed facet $\cF\subset\overline{C}$, we have 
$$W_\cF \cap (Z/(Z\cap I))\subset W_\aff \cap \Omega = \{1\},$$ 
and therefore $\tW_\cF\cap \{\widehat{z}\}_{z\in Z/(Z\cap I)} = \{1\}$.  This implies that the composition
$$\cH_{\cF}\longhookrightarrow \cH_{\cF}^{\dagger}\longtwoheadrightarrow \cH_{\cF}^{\dagger,\zeta}$$
is injective, and we may view $\cH_{\cF}$ as a subalgebra of $\cH_{\cF}^{\dagger,\zeta}$.  Similarly to Lemma \ref{daggerproj}, we see that $\cH_{\cF}^{\dagger,\zeta}$ is free of finite rank over $\cH_{\cF}$, with basis given by $\{\overline{\T_{\widehat{\omega}}}\}_{\omega\in \Omega_{\cF}/(Z/(Z\cap I))}$, so that $\cH_{\cF}^{\dagger,\zeta}$ has the structure of a crossed product algebra over $\cH_{\cF}$.

Now let $\mm$ be a right $\cH$-module on which $\cY$ acts by the character $\zeta$, so that $\mm$ is naturally a module over $\cH^\zeta$.  The inclusion $\cH_{\cF}^\dagger \subset \cH$ induces an inclusion $\cH_{\cF}^{\dagger,\zeta}\subset \cH^\zeta$ for every facet $\cF\subset \overline{C}$, which makes $\cH^\zeta$ into a free (left and right) $\cH_{\cF}^{\dagger,\zeta}$-module.  We easily see that the natural map $\cH\longtwoheadrightarrow \cH^\zeta$ induces an isomorphism of $\cH^\zeta$-modules
\begin{equation}\label{zchar}
\mm|_{\cH_{\cF}^\dagger}(\epsilon_{\cF})\otimes_{\cH_{\cF}^\dagger}\cH \cong \mm|_{\cH_{\cF}^{\dagger,\zeta}}(\epsilon_{\cF})\otimes_{\cH_{\cF}^{\dagger,\zeta}}\cH^\zeta.
\end{equation}
Viewing \eqref{resn} as a resolution of $\cH^\zeta$-modules, we obtain the following result.

\begin{propn}\label{zcharprop}
Assume $p\nmid |\Omega/(Z/(Z\cap I))|$, and let $\mm$ be a right $\cH$-module on which $\cY$ acts by a character $\zeta$.  Then $\pd_{\cH}(\mm) < \infty$ if and only if $\pd_{\cH^\zeta}(\mm) < \infty$.  In this case, \eqref{resn} is a resolution of $\mm$ by projective $\cH^\zeta$-modules, and we have $\pd_{\cH^\zeta}(\mm) = r_{\textnormal{ss}}$ and $\pd_{\cH}(\mm) = r_{\textnormal{ss}} + r_{Z}$.  
\end{propn}

Compare \cite[Cor. 2.3]{opdamsolleveld}.  

\begin{proof}
Assume first that $\pd_{\cH}(\mm) < \infty$.  The condition on $p$ implies $p\nmid |\Omega_{\tor}|$, so that $\mm|_{\cH_{\cF}}$ is projective over $\cH_{\cF}$ for every facet $\cF\subset \overline{C}$ (Lemma \ref{redtofacets}).  Since $\cH_{\cF}^{\dagger,\zeta}$ is a crossed product algebra over $\cH_{\cF}$, \cite[Thm. 7.5.6(ii)]{mcconnellrobson} again implies that $\mm|_{\cH_{\cF}^{\dagger,\zeta}}(\epsilon_{\cF})$ is projective over $\cH_{\cF}^{\dagger,\zeta}$ (cf. Lemmas \ref{daggerproj} and \ref{twist}).  Since $\cH^\zeta$ is free as an $\cH_{\cF}^{\dagger,\zeta}$-module, the induced $\cH^\zeta$-module $\mm|_{\cH_{\cF}^{\dagger,\zeta}}(\epsilon_{\cF})\otimes_{\cH_{\cF}^{\dagger,\zeta}}\cH^\zeta$ is projective, which gives one implication and the precise result about the resolution \eqref{resn}.

Assume conversely that $\pd_{\cH^\zeta}(\mm) < \infty$.  By an analog of Lemma \ref{redtofacets}, we get that $\mm|_{\cH_{\cF}}$ is projective over $\cH_{\cF}$ for every facet $\cF\subset\overline{C}$, and using the argument of the first paragraph shows that in fact $\pd_{\cH^{\zeta}}(\mm) \leq r_{\textnormal{ss}}$.  Also, \cite[Prop. 5.4]{ollivierschneider} shows that the algebras $\cH_{\cF}^{\dagger,\zeta}$ are Frobenius algebras, and by adapting the arguments of Section 6 of \emph{loc. cit.} we see that \eqref{resn} is a resolution of $\mm$ by (Gorenstein) projective $\cH^\zeta$-modules (using the identification \eqref{zchar}).  Furthermore, an analog of Lemma 6.10 of \emph{loc. cit.} holds, and we obtain 
$$\textnormal{Ext}^{r_{\textnormal{ss}}}_{\cH^\zeta}(\mm,\cH^\zeta)\neq 0.$$
Thus, we have $\pd_{\cH^\zeta}(\mm) = r_{\textnormal{ss}}$.

Fix now a set of generators $z_1,\ldots, z_{r_Z}$ for $Z/(Z\cap I)$; then $\T_{\widehat{z}_i} - \zeta(\T_{\widehat{z}_i})$ are central non-zerodivisors of $\cH$ which generate the (proper) ideal $\left(\T_{\widehat{z}} - \zeta(\T_{\widehat{z}})\right)_{z\in Z/(Z\cap I)}\cH$.  Applying \cite[Thm. 7.3.5(i)]{mcconnellrobson} shows that $\pd_{\cH}(\mm) = \pd_{\cH^\zeta}(\mm) + r_Z = r_{\textnormal{ss}} + r_Z$, which gives the claim.  
\end{proof}

\begin{remark}
Let $\mm$ and $\nn$ be two right $\cH$-modules on which $\cY$ acts by a character $\zeta$.  In order to obtain quantitative information about projective dimensions, one can compute the Ext groups between $\mm$ and $\nn$, either in the category of $\cH^\zeta$-modules or in the category of $\cH$-modules.  The relation between the two is controlled by the base-change-for-Ext spectral sequence:
$$E_2^{i,j} = \Ext_{\cH^\zeta}^i(\mm,\Ext_{\cH}^j(\cH^\zeta,\nn))\Longrightarrow \Ext_{\cH}^{i + j}(\mm,\nn).$$
The five-term exact sequence associated to the above spectral sequence is a module-theoretic version of a short exact sequence used by Pa\v{s}k\={u}nas in \cite[Prop. 8.1]{paskunas:exts}.
\end{remark}

\bibliographystyle{amsplain}
\bibliography{refs}

\end{document}